\documentclass[11pt]{amsart}
\usepackage{amsfonts}
\usepackage{geometry}                
\usepackage{graphicx}
\usepackage{amssymb}
\usepackage{epstopdf}
\usepackage{setspace}
\usepackage{leftidx}

\usepackage{xcolor}

\theoremstyle{plain}
\newtheorem{theo}{Theorem}[section]

\newtheorem{cor}[theo]{Corollary}
\newtheorem{prop}[theo]{Proposition}

\numberwithin{equation}{section}

\theoremstyle{definition}

\newtheorem{remark}[theo]{Remark}
\newcommand{\bequ}{\begin{equation}}
\newcommand{\eequ}{\end{equation}}
\newcommand{\bali}{\begin{align}}
\newcommand{\eali}{\end{align}}

\def\e{\varepsilon}

\def\vf{\varphi}

\def\b{\beta}

\def\d{\delta}

\def\D{\Delta}
\def\G{\Gamma}
\def\g{\gamma}
\def\k{\kappa}

\def\LL{\mathcal L}
\def\pp{\partial}
\def\l{\lambda}

\def\s{\sigma}

\def\x{\times}
\def \R{\mathbb R}
\def \N{{\mathbb N}}
\def\E{\mathbb E}
\def \Z{\mathbb Z}

\def \H{\mathbb H}

\def \P{\mathbb P}
\def\ov{\overline}
\def\un{\underline}

\def\U{\underbar}
\def\om{\omega}
\def\Om{\Omega}

\def\U{\mathcal U}
\def\C{\mathbb C}

\def\W{\mathcal W}
\def\wh{\widehat}
\def\wt{\widetilde}
\def\({\biggl(}
\def\){\biggr)}

\def\<{\bold\langle}
\def\>{\bold\rangle}

\def\LL{{{\mathcal L}}}

\def\M{\widetilde {M}}

\def\vol{{\rm{Vol}}}

\usepackage{mathtools,mathabx}
\newcommand\AR[1]{\makebox[0pt][l]{$#1$}\kern0.5em\raisebox{1.5ex}{$\curvearrowleft$}} 

\usepackage{graphicx}
\usepackage{float}
\usepackage{stmaryrd}

\begin{document}
\title[Entropies for negatively curved manifolds]{Entropies for negatively curved manifolds}
\author{Fran\c cois Ledrappier  and  Lin Shu}
\address{Fran\c cois Ledrappier,  Sorbonne Universit\'e, UMR 8001, LPSM, Bo\^{i}te Courrier 158, 4, Place Jussieu, 75252 PARIS cedex
05, France}\email{fledrapp@nd.edu}
\address{Lin Shu,  LMAM, School
of Mathematical Sciences, Peking University, Beijing 100871,
People's Republic of China} \email{lshu@math.pku.edu.cn}

\thanks{The second author was  partially  supported by  NSFC (No.11331007 and No.11422104).}

\maketitle

This is a survey of several notions of entropy related to a compact manifold of negative curvature and of some relations between them. Namely, let $(M, g)$ be a $C^ \infty$ compact boundaryless Riemannian connected manifold with negative curvature. After recalling the basic definitions, we will define and state the first properties of 
\begin{enumerate}
\item the volume entropy $V$,
\item the dynamical entropies  of the geodesic flow, in particular the entropy $H$ of the Liouville measure and the topological entropy (which coincides with $V$), 
\item the stochastic entropy $h_\rho$ of a family of (biased) diffusions related to the stable foliation of the geodesic flow, \item the relative dynamical entropy of natural  stochastic flows representing the (biased) diffusions.
\end{enumerate}

Most of the results in this survey are not new, some are classical, and we apologize in advance for any inaccuracy in the attributions. New observations are Theorems  \ref{cute} and \ref{cuter}, but the main goal of this survey is to present together related notions that are spread out in the literature. In particular, we are interested in the different so-called rigidity results and problems that (aim to) characterize locally symmetric spaces among negatively curved manifolds by equalities in general entropy inequalities.

These notes grew out from lectures delivered by the second author in the workshop {\it {Probabilistic methods in negative curvature}} in ICTS, Bengaluru, India, and we thank Riddhipratim Basu, Anish Ghosh and Mahan Mj for giving us this opportunity. We also thank  Nalini Anantharaman, 
Ashkan Nikeghbali  for organizing  the 2nd  Strasbourg/Zurich Meeting on {\it {Frontiers in Analysis and Probability}} and Michail Rassias for allowing us to publish these notes that have only a loose connection with the talk of the first author there.

\

\section{Local symmetry and volume growth}

 Let $(M, g)$ be a $C^ \infty$ compact boundaryless connected $d$-dimensional Riemannian manifold and for $u,v$ vector fields on $M$ we denote $\nabla _u v$ the covariant derivative of $v$ in the direction of $u$. Given $u,v \in T_xM$, the {\it {curvature tensor $R$}}
   associates to a vector $w \in T_x M$ the vector $R(u,v) w $ given by 
\[ R(u,v)w \; = \; \nabla _u\nabla _v w -\nabla _v\nabla _u w - \nabla _{[u,v]} w . \]
The space $(M, g)$ is called {\it{locally symmetric}} if $\nabla R=0$.

Consider the case $(M, g)$ has {\it{negative sectional curvature}}, i.e.,  for  all  non colinear $u,v \in T_x M,$ $x\in \M$, the  {\it sectional curvature  $\displaystyle K(u,v):= \frac{<R(u,v)v,u>}{|u\wedge v|^2}$} is negative.  Connected simply connected locally  symmetric spaces of negative sectional curvature are non-compact. They have been classified and are one of the hyperbolic spaces  $\H_\R^n, \H_\C^n, \H_\H^n, \H_\mathbb O^2$, respectively of dimension respectively $n, 2n, 4n, 16$. Hyperbolic spaces are obtained as quotients  of  semisimple Lie groups of real rank one (respectively $SO(n,1), SU(n,1), Sp(n,1), F_{4(-20)}$), endowed with  the  metrics coming from the Killing forms, by  maximal compact subgroups. By general results of Borel (\cite{Bl}) and Selberg (\cite{Se}), these spaces admit compact boundaryless quotient manifolds and those locally symmetric $(M,g_0)$ are the basic examples of our objects of study.  Clearly, $C^2$ small $C^\infty $ perturbations of $g_0$ on the same space $M$ yield  other examples of compact negatively curved  manifolds.  Different examples  of non-locally symmetric, compact, negatively curved  manifolds  have been constructed (see \cite{MS}, \cite{D}, \cite {GT}, \cite {FJ}). They are supposed to be abundant, even if constructing explicit ones is often delicate. 

It is natural to ask if we can  recognize locally symmetric spaces through global properties or  quantities.   One supportive example is the volume entropy. Let $\M$ be the universal cover space of $M$ such that $M=\M/\Gamma$, where $\Gamma:=\Pi_1(M)$ is the fundamental group of $M$, and endow $\M$ with metric $\wt{g}$, which is the $\G$-invariant extension of $g$. The volumes on $(M, g)$ and $(\M, \wt{g})$ are denoted ${\rm Vol}_{g}$ and ${\rm Vol}_{\wt{g}}$, respectively.  (We will fix  a connected fundamental domain  $M_0$ for the  action of $\Gamma$ on $\M$. The restriction of ${\rm Vol}_{\wt{g}}$ on $M_0$ is also denoted ${\rm Vol}_{g}$.)  For $x\in \M$, let $B_{\M}(x, r), r>0,$ denote the ball centered at $x$ with radius $r$.  The following limit exists (independent of $x \in \M$) and defines the {\it{volume entropy}} (Manning, \cite{Mg}):
\begin{equation*}
 V(g) : =  \lim\limits_ {r \to +\infty} \frac{1}{r} \log \vol_{\wt{g}} B_{\M}(x, r). \end{equation*}
Since $(M, g)$ is negatively curved, by Bishop comparison theorem, $V(g)>0$. The following rigidity result is shown by Besson-Courtois-Gallot (\cite{BCG}):

\begin{theo}[\cite{BCG}]\label{BCG95} Let $(M,g_0) $ be closed locally symmetric space of negative curvature,  and consider another metric $g$ on $M$ with negative curvature and such that  $\vol _g (M) = \vol _{g_0} (M) $. Then, 
\[V(g) \; \geq V(g_0).\] If $d = dim (M) >2$, one has equality only if $(M,g)$ is isometric to $(M,g_0)$.\end{theo}

 If $ d = 2,$ equality holds if, and only if, the curvature is constant (Katok, \cite{Kk}). In the case $d>2$, Katok (\cite{Kk}) proved  Theorem \ref{BCG95} under the hypothesis that $g $ is conformally equivalent to $g_0$. 
 \begin{remark} The theorem holds even if $g'$ is a metric on another manifold $M'$, homotopically equivalent to $M.$ \end{remark}

The locally symmetric property can also be interpreted as  geodesic symmetry. A {\it {geodesic}} in $M$ is a curve $t \mapsto \g(t), t \in \R,$ such that, if $\dot \g (t) := \frac{d}{dt} \g (s) \big|_{s=t}$, satisfies $\nabla _{\dot \g(t)} \dot \g(t)=0$ for all $t$.    For all $v \in TM, $ there is a unique geodesic $\g_v (t) $ such that $\dot \g (0) = v$. The {\it {exponential map}} $\exp_x: T_x M \to M$ is given by $\exp_x v = \g_v (1). $ By compactness, there exists $\iota>0 $ such that, for all $x \in M, \, \exp_x $ is a diffeomorphism between the  ball of radius $\iota$ in $(T_xM, g_x)$ and the ball of radius $\iota$ about $x$ in $M$.  The Cartan-Ambrose-Hicks Theorem implies that the space is locally symmetric if, and only if,  for any $x \in M$, the geodesic symmetry about $x$ defined by $ y \mapsto \exp_x (- \exp_x^{-1} y ) $ is a local isometry. 

One natural dynamics related to geodesics is the geodesic flow.  Let $SM := \{ v, v \in TM : \| v \| = 1 \}$ be the unit tangent bundle. The {\it {geodesic flow }} $\vf _t $ on $SM$ is such that $\vf _t (v) = \dot \g_v(t)$ for $ t \in \R.$ Denote $\ov X(v) \in T_vSM$ the vector field on $SM$ generating the geodesic flow. The derivative $D_v \vf _t$ is described using {\it {Jacobi fields.}} Let $s \mapsto v(s) $ be a curve in  $SM$ with $v(0) = v, \dot v(0) = w \in T_vSM. $ Then, $s\mapsto \g_{v(s)}(t) $ is a curve with tangent vector $J(t) $ at $\g_v(t).$
$J(t) $ satisfies the {\it {Jacobi equation:}}
\begin{equation}\label{Jacobi} \nabla _{\dot \g }  \nabla _{\dot \g } J(t) + R (J(t), \dot \g (t)) \dot \g (t) = 0. \end{equation}
\begin{proof} By definition, \[ R (J(t), \dot \g (t)) \dot \g (t) = \nabla _{J(t)}  \nabla _{\dot \g(t) } \dot\g (t) - \nabla _{\dot \g(t) }  \nabla _{J(t)}\dot \g(t) - \nabla _{[J(t), \dot \g(t) ]} \dot \g .\]
 We have $\nabla _{\dot \g(t) } \dot\g (t) =0$ by definition, $[J(t), \dot \g(t) ] = [\frac{\pp}{\pp s} , \frac{\pp}{\pp t} ] = 0 $ and so $ \nabla _{J(t)}\dot \g(t) =  \nabla _{\dot \g(t)}J(t) $ (we use the fact that $\nabla_uv-\nabla _v u = [u,v]$). \end{proof}

 We will consider $C^ \infty $ compact boundaryless connected  Riemannian manifolds with  negative sectional curvature.  It follows  from (\ref{Jacobi}) that  $t \mapsto \| J(t) \|^2 $ is a strictly convex function (by a direct computation). In particular,  $\exp_x$ is a diffeomorphism  from $T_xM $ to  the universal cover $\M$.  Two geodesic rays $\gamma_1, \gamma_2$ in $\M$ are said to be equivalent if $\sup _{t\geq0 } d(\g_1(t), \g_2(t)) < \infty .$ The space of equivalence classes $\partial\M:=\{[\g _v (t), t\geq 0], v\in TM\}$ is the {\it {geometric  boundary at infinity}}.
For $x \in \M, \pi _x: S_x\M \to \pp \M, \pi _x (v) = [\g _v (t), t\geq 0]$ is one-to-one ($\pi _x $ is injective by convexity (of $t\mapsto d(\g_v(t), \g_{w}(t))$  for $w\in S_x\M$ with $w\not=v$)  and for any geodesic $\g(t)$, any $t>0$, one can find $v_t \in S_x\M$ such that $\g(t) \in \g_{v_t} (s), s\geq 0$; any limit point $v$ of $v_t, t\to +\infty,$ is such that $\g_v$ is equivalent to $\g$).  Thus, the unit tangent bundle $S\M$ is identified with $\M \x \pp \M.$  For any two points $\xi , \eta$  in $\pp \M$, there is a unique geodesic $\g_{\eta, \xi} $ (up to time translation) such that $\g_{\eta, \xi}  (+ \infty ) :=\lim_{t\to +\infty}\g_{\eta, \xi}(t)= \xi$ and $\g_{\eta, \xi}  (-\infty ) :=\lim_{t\to -\infty}\g_{\eta, \xi}(t) = \eta $. The topology on $\M \x \pp \M$ is such that two pairs $(x,\xi )$ and $(y,\eta ) $ are close if $x$ and $y$ are close and the distance from $x$ to the  geodesic $\g_{\eta, \xi} $ is large. The group $\G$ acts discretely and cocompactly on $\M$. The action of $\G$ extends continuously to $\pp \M$ and the diagonal action of $\G$ on $\M \x \pp \M$ is again discrete and cocompact. The quotient $(\M \x \pp \M )/\G= S\M /\G$  is identified with $SM$.

We continue  to use $\varphi_t$ to denote the geodesic flow on $S\M$. It has the {\it{Anosov property}} (\cite{Ano}): each $\varphi_t, t\not=0$,  has no fixed point and there is a continuous decomposition $\{T_vS\M=E^{ss}(v)\oplus \ov{X}(v)\oplus E^{su}(v), v\in S\M\}$ with $\ov{X}(v)$ being the geodesic spray tangent to the flow direction and constants $C$, $C>0$, $\l$,  $\l\in (0, 1)$,  such that,  for $t>0$, 
\[
\|D_v\varphi_t w_s\|\leq C\l^t\|w_s\|, \ \forall w_s\in E^{ss}(v),\ \|D_v\varphi_{-t} w_u\|\leq C\l^t\|w_u\|, \ \forall w_u\in E^{su}(v).
\]
For $v=(x, \xi)\in S\M$, the 
{\it {stable manifold  at $v$}} of the geodesic flow, \[ \wt W^s(v) \; := \; \big\{ w: \sup _{t \geq 0} d(\vf _t w, \vf _t v) < + \infty \big\}\]
is  tangent  to $E^{ss}(v)\oplus \ov{X}(v)$.  The $\wt W^s(v)$ can be identified with $\M \x \{\xi\}$ and hence is endowed naturally with the metric $\wt g$. The quotient $( \M \x \{\xi\} )/\G $ is the {\it {stable manifold }} $W^s (v). $ As $\xi $  varies, they form a H\"older continuous lamination  $\W^s$ of $SM$ into $C^\infty $ manifolds of dimension $d$ which is called the {\it {stable foliation}}. Therefore, the  metric on each individual stable manifold  comes from the local identification with $\M.$ The {\it {strong stable manifold at $v$}}, 
\[ \wt W^{ss} (v) := \big\{ (y,\xi ) : \lim\limits_{t \to +\infty} d(\g_{x, \xi} (t), \g_{y, \xi} (t)) = 0 \big\} \]
has tangent $E^{ss}(v)$. 
Let $\underline{v}$ be the projection of $v$ on $SM$; then, $\wt W^{ss} (v) $ projects  onto  
\[W^{ss} (\underline{v}) := \big\{ w \in SM: \lim\limits_{t \to +\infty} d(\g_w(t), \g_{\underline{v}} (t)) = 0 \big\}.\]
The collection of $\{W^{ss} (\underline{v}), \underline{v}\in SM\}$ form a H\"older continuous lamination $\W^{ss}$   of $SM$ into $C^\infty $ manifolds of dimension $d-1$ which is called the {\it {strong stable foliation}}. 

  For $v=(x, \xi ) \in S\M$, define the {\it{Busemann function}} \[ b_{x, \xi } (y) =  b_{x, \xi } (y,\xi )\;: = \; \lim\limits_{z\to \xi} \left( d(y,z ) - d(x,z) \right), \ \forall y\in \M. \]
The level set $\{ (y,\xi ) :  b_{x, \xi } (y, \xi ) = 0 \} $ coincides with  $\wt W^{ss}(x,\xi)$ and the set of  its foot points is the  horosphere of $(x, \xi)$. Denote ${\textrm {Div}}^s, \nabla ^s$ the divergence and gradient  along $\wt W^s$ (and $W^s$) 
induced by the metric $\wt g$ on $\M \x \{\xi\},$   $\D ^s = {\textrm {Div}}^s \nabla ^s.$   Then, 
\[ \nabla _y  b_{x, \xi } (y) |_{y=x} = - (x,\xi) \  {\textrm {or}} \ \nabla ^s_w b_v (w)|_{w=v} = -\ov X(v).\]
Set 
\begin{equation*}
\quad B(x, \xi) := \D_y b_{x,\xi}(y) |_{y=x} = - {\textrm {Div}}^s \ov X (v).
\end{equation*}
Geometrically, the  $B(x, \xi)$ is the mean curvature at $x$ of the horosphere of $(x, \xi)$. 
The function $B$ is a $\Gamma$-invariant function on $S\M$. We still denote $B$ the function on the  quotient $SM$. From the definition follows:  
\begin{equation}\label{jacobian} B(v) = - \frac{d}{dt} \log {\textrm {Det}} D_v \vf_t |_{W^{ss}(v)} \big |_{t=0}.\end{equation}
So, dynamically, $-B$ tells the exponential  growth rate of the volume on $W^{ss}$ under the geodesic flow $\varphi_{t}$, $t>0$. 
 It follows from (\ref{jacobian}) that the function $B$ is H\"older continuous on $SM$.  The main property of the function $B$ is the following, whose proof combines the works of Benoist-Foulon-Labourie (\cite{BFL}), Foulon-Labourie (\cite{FL}) and Besson-Courtois-Gallot (\cite{BCG}).
\begin{theo}[\cite{BFL, FL, BCG}] \label{BCG} The function $B$ is constant if, and only if, the space $(M,g)$ is locally symmetric.\end{theo}

\begin{remark} There is a positive operator $U$ on the orthogonal space to $v$ in $T_xM$  satisfying the Riccati equation $\dot U + U^2 + R( \cdot, \dot \g (t)) \dot \g (t) = 0$ and  such that $B = {\textrm {Tr}}U.$   If $d=2$, the equation reduces to $\dot B + B^2 +K = 0.$ Clearly, if $B$ is constant, then the curvature $K$ is the constant $-B^2.$ If $d=3$, one can also conclude from the Riccati equation and some matrix calculations that $B$ is constant if, and only if, the sectional curvature  is constant (see Knieper \cite{Kn}).  \end{remark}

\

\section{Dynamical entropy and an application of thermodynamical formalism}

More quantities related to $V, B$ can be introduced through a dynamical point of view.

\subsection{Dynamical entropy}
Let $T $ be a continuous transformation of a compact metric space $X$. For $x \in X, \e >0, n \in \N,$  define the {\it {Bowen ball }} $B(x, \e, n)$
\[ B(x,\e, n) := \{ y \in X: d(T^jy, T^jx) <\e {\textrm { for }} 1\leq j\leq n \}\]
and the {\it{entropy}} $h_m(T) $ of a $T$-invariant probability measure $m$
\[ h_m(T) := \sup _\e \int \left( \limsup_n -\frac{1}{n} \log m (B(x,\e,n) )\right) \, dm(x) .\]
It is easy to see that for $j \in \Z, h_m (T^j ) = |j| h_m (T).$ A useful upper bound of $h_m(T)$ is given by Ruelle inequality (\cite{Re}) using the average maximal exponential growth rate of all the parallelograms under the iteration of the tangent map  $DT$. 
\begin{theo}[Ruelle, \cite{Re}] Assume $X$ is a compact manifold and $T$ a $C^1$ mapping of $X$. Then, for any $T$-invariant probability measure $m$,
\[ h_m (T) \; \leq \;  \int \left(  \sup _k  \limsup_n \frac{1}{n} \log \| \wedge ^k  D_x T^n \| \right) \, dm(x),\]
where $\wedge ^k  D_x T^n$ denotes  the $k$-th exterior power of $D_x T^n$.\end{theo}
\begin{cor} \label{Ruelle} If $X = SM$, where $(M, g)$  is a compact, boundaryless, $C^2 $ Riemannian manifold with negative sectional curvature and dimension $d$,  $m$ a geodesic flow invariant probability measure, and $t\in \R$, \[ h_m (\vf _t ) \;\leq\; |t| \int_{SM}  B \, dm .\] \end{cor}
\begin{proof} For $ v\in SM, t<0,  |t| $ large, the highest value of $\| \wedge ^k  D_v \vf _t \| $ is obtained for $k =d-1$ and is the Jacobian of $D_v \vf _t $ restricted to $T_v W^{ss} (v).$   By (\ref{Jacobi}), this is $e^{\int_t^0 B(\vf _s v)\, ds}$. By the ergodic theorem, \[ \lim\limits _{n\to +\infty} \frac{1}{n} \log \big\| \wedge ^{d-1}  D_v \vf _{nt}|_{W^{ss}} \big\| = \lim\limits _{n\to +\infty} \frac{1}{n} \int_{nt}^0 B(\vf _s v)\, ds \] exists and has integral $ |t| \int B \,dm.$ The conclusion follows by Ruelle inequality.\end{proof}

Another general inequality is given by 

\begin{theo}[Manning, \cite{Mg}] \label{Manning} Let  $(M, g)$  be  a compact, boundaryless, $C^2 $ Riemannian manifold with negative sectional curvature and dimension $d$,  $m$ a geodesic flow invariant probability measure, and $t\in \R$, \[ h_m (\vf _t ) \;\leq\; |t| V.\] \end{theo}

\begin{remark}\label{Manning-Lemma} The proof of Theorem \ref{Manning} is based on the following consequence of nonpositive curvature (\cite{Mg}, Lemma page 571). 
For any $v,w \in SM$, any $r \geq 1,$
\begin{eqnarray*}\max \{\sup_{0\leq s\leq 1} d(\vf_s v, \vf_s w), \sup_{r-1\leq s\leq r} d(\vf_s v, \vf_s w)\} &\leq & \sup_{0\leq s\leq r} d(\vf_s v, \vf_s w)\\
&\leq & \sup_{0\leq s\leq 1} d(\vf_s v, \vf_s w) + \sup_{r-1\leq s\leq r} d(\vf_s v, \vf_s w).
\end{eqnarray*}
This observation can also be used to give a direct proof of Corollary \ref{Ruelle}. \end{remark}

\subsection{Thermodynamical formalism} For simplicity, we introduce the notion of pressure  by  the classical variational principle. 
Let $(X,T)$ be a continuous mapping of a compact metric space.
The {\it {pressure}} $P(F) $ of a  continuous function $F: X \to \R$ is defined by 
\[ P(F) \; := \; \sup _m \left \{h_m(T) + \int F \, dm\right \},\] where $m$ runs over all $T$-invariant probability measures. Let $X= SM$ be closed negatively curved and $T=  \vf _1.$ From Ruelle and Manning inequalities follow
\[ P(-B) \leq 0 \quad {\textrm{and}} \quad P(0 ) \leq V.\]
We will construct later  the {\it {Liouville measure }} $m_L$ with the property  (Theorem \ref{Pesin}) 
\begin{equation}\label{entropy-Liouville}h_{m_L}(\vf _1) = \int B \, dm_L = : H
\end{equation}
 and the {\it{ Bowen-Margulis measure}} $m_{BM} $ such that (Theorem \ref{BM-measure})
\begin{equation}\label{entropy-BM}
h_{m_{BM}} (\vf _1) = V.\end{equation}
This will show that $P(-B) = 0 $ and $P(0)= V$. Using these properties, we can prove:

\begin{theo}\label{cute}  Let $(SM, \vf_t)$ be the geodesic flow on a closed manifold of negative curvature. Let $\mathcal M$ be the set of $\vf_t$-invariant probability measures, $H$ and $V$ as defined above. Then, 
\begin{equation}\label{inf-sup-inequality} \inf _{m \in \mathcal M} \int B \, dm  \leq H \leq V \leq \sup_{m \in \mathcal M} \int B \, dm,
\end{equation}
with equality in one of the inequalities if, and only if, $m_L = m_{BM}.$ Moreover, in that case, $\int B \, dm = V$ for all  $ m \in \mathcal M$. \end{theo}

\begin{proof} Since the function $B$ is H\"older continuous on $SM$, for each $s\in \Bbb R$, there exists a unique invariant probability measure  (equilibrium measure for $sB$) $ {\rm P}(s):= P(sB)= h_{m_s}(\vf _1) + s  \int B \, dm_s$ (\cite[Proposition 4.10]{PP}).\footnote{Chapter 4 in \cite{PP} is only concerned with subshifts of finite type. The extension of \cite{PP} Propositions 4.8, 4.9, 4.10 to suspended flows  is direct (see \cite{PP}, Chapter 6) and the application to geodesic flows on compact negatively curved manifolds is standard (cf. \cite{PP}, Appendix 3).}  For example, by (\ref{entropy-Liouville}), (\ref{entropy-BM}), $m_{L}$, $m_{BM}$ are equilibrium measures for $-B$ and $0$ respectively.  Together with  
Corollary \ref{Ruelle}, we obtain
\[
\inf _{m \in \mathcal M} \int B \, dm\leq \int B \, dm_L=H\leq \sup_{m \in \mathcal M}\{h_m(\varphi_1)\}=V\leq  \int B \, dm_{BM}\leq \sup_{m \in \mathcal M} \int B \, dm,
\]
which gives (\ref{inf-sup-inequality}).

 Clearly, using the uniqueness of $m_s$, we have $H=V$ if,  and only if $m_L = m_{BM}$.  To show any equality in  the other inequalities  of (\ref{inf-sup-inequality}) holds  if, and only if, $m_L = m_{BM},$ we use  properties of the pressure function, in particular of the convex function $s \mapsto {\rm P}(s)$. We already know that ${\rm P}(-1) = 0 $ and that ${\rm P}(0) = V.$ From the definition follows that $\inf _{m \in \mathcal M} \int B \, dm $ and $\sup _{m \in \mathcal M} \int B \, dm $ are the slopes of  the asymptotes of the function ${\rm P}(s)$ as $s \to -\infty $ and $+\infty $ respectively. Since the function $B$ is H\"older continuous on $SM$, the function $s \mapsto {\rm P}(s) $  is real analytic (\cite[Proposition 4.8]{PP}). Moreover, the slope at $s$ is given by $\int B\, dm_s$ (\cite[Proposition 4.10]{PP}). Now, if $H  =\inf _{m \in \mathcal M} \int B \, dm ,$ the function  $s \mapsto {\rm P}(s) $ is affine on $[-\infty ,-1]$ and thus everywhere. Since the slopes of ${\rm P}(s)$ at $-1$   and $0$  are $\int B \, dm_{L}=H$  and $\int B \, dm_{BM},$ respectively, and $H\leq V\leq \int B \, dm_{BM}$, hence we must have $V =\int B \, dm_{BM}$, which implies that $m_{BM} $ coincides with  $m_L$ and $V= H$. Finally, if $V = \sup_{m \in \mathcal M} \int B \, dm,$ the measure $m_{BM} $ is the equilibrium measure for $-B$, which must coincide with $m_{L}$.  
 
Assume $m_{BM} $ and $m_L$ coincide, then (\cite[Proposition 4.9]{PP}), there exists a continuous function $F$ on $SM$, $C^1$ along the trajectories of the geodesic flow, such that 
\[  - B  = {\rm P}(-1) - {\rm P}(0) + \frac{\pp }{\pp t } F\circ \vf _t \big|_{t=0}.\] 
In particular,   $\int B \, dm = P(0) = V$ for all  $ m \in \mathcal M$.
\end{proof}

\subsection{Liouville measure}

For $x \in \M $, let $\l_x$ denote the pull back  measure on $\partial\M$  of the Lebesgue probability measure on $S_x\M$ through the mapping $\partial\M\mapsto S_x\M:\ \xi\mapsto (x, \xi)$.  Define a measure $\wt m_L$ on $\M\x \pp \M$ by setting
\[ \int F (x,\xi )\, d\wt m_L  =   \int _{\M} \left(\int _{\pp \M} F(x, \xi) \, d\l_x (\xi) \right)  \frac{ d\vol_{\wt{g}} (x)}{\vol_{g} (M)}.\]
It is clear from the definition that the measure $\wt m_L$ is $\G$-invariant. There is  a $D\vf_t$-invariant  2-form on $\ov X^\perp $ in $TSM$ defined by the Wronskian $\W$
 \[ \W \big( (J_1, J'_1), (J_2, J'_2)\big) \; := \; <J_1(t), J'_2(t)> - <J'_1(t), J_2(t)>. \] Assume $M$ is orientable. The $(2d-1)$-form  $\wedge^{d-1}\W \wedge dt $ is  nondegenerate  and invariant. 
 For $v \in SM$, take a positively oriented  orthonormal basis $\{e_0, \cdots , e_{n-1} \}$  in $T_xM$ such that $e_0 = v$. By computing $\wedge^{d-1}\W \wedge dt $ on the $(2d-1)$-vector
 $\left((e_1, 0), (0,e_1), \cdots, (e_{n-1},0),(0, e_{n-1}), \ov X \right),$ one sees that the measure associated to this volume form is the one we defined. So the measure $\wt m_L$ is invariant  under the geodesic flow.  We do the same computation on a double  cover of $M$ if $M$ is not orientable.
 
  The measure $m_L $ on $SM$ that extends to $\wt m_L $ is a $\vf_t$-invariant probability measure which is called the {\it{Liouville probability measure}.}  It satisfies
\begin{theo}\label{Pesin} For all $t\in \R, \; \;  h_{m_L} (\vf _t ) =|t| \int B \, dm_L.$ \end{theo}
\begin{proof}({\it {Sketch}}) It suffices to prove the theorem for $t =-1.$ In the definition of entropy, we can use the flow Bowen balls ${\bf B}(v,\e, r)$, $\e, r>0$, 
\[{\bf B}(v,\e, r) := \{ w :   \sup_{-r\leq s\leq 0} d(\vf_s v, \vf_s w) < \e \}.\] 
By Remark \ref{Manning-Lemma},\[ {\bf B}(v,\e/2, 1) \cap \vf _{r-1} {\bf B}(\vf _{-r+1}v,\e/2, 1) \subset {\bf B}(v,\e, r) \subset {\bf B}(v,\e, 1) \cap \vf _{r-1} {\bf B}(\vf _{-r+1}v,\e, 1).\]
Estimating the Liouville  measure of $ {\bf B}(v,\e, 1) \cap \vf _{r-1} {\bf B}(\vf _{-r+1}v,\e, 1)$ reduces to estimating the $d$-dimensional measure of $ B^s(v,\e) \cap \vf _{r-1} B^s(\vf _{-r+1}v,\e), $ where $B^s (v, a) $ is the ball of radius $a$ and center $v$ in $W^s (v)$. It follows from (\ref{jacobian}) that this measure is, up to error terms that depend on $\e$ small enough, but not on $r$, equal to \[ \textrm{Det}  D_{\vf _{-r+1} v} \vf_r |_{W^{s}(\vf _{-r+1}v)} = e^{-\int_{-r+1}^0 B(\vf _s v) \, ds}.\]
It follows that, if one takes $\e $ small enough,
\[ h_{m_L} (\vf _{-1}) = \lim\limits _{r\to +\infty } \frac{1}{r} \int_{SM} \left( \int_{-r+1}^0 B(\vf _s v) \, ds \right) \, dm_L (v) \; = \; \int_{SM} B \, dm_L.\]\end{proof}
Observe that, since $m_L$ is a measure realizing the maximum in  $P(-B)$, it is ergodic.
\begin{remark} Basic facts about ergodic theory and thermodynamic formalism are in Bowen (\cite{B2}); see also Parry-Pollicott (\cite{PP}). The definition of the entropy given here is due to Brin-Katok (\cite{BK}). The ergodicity of $m_L$ with respect to the geodesic flow is a landmark result of Anosov (\cite{Ano}).
\end{remark}
\

\section{ Patterson-Sullivan, Bowen-Margulis, Burger-Roblin}

In analogy to the construction of the measure $m_{L}$, one can obtain the Bowen-Margulis measure $m_{BM}$ using a class of measures (Patterson-Sullivan measures) on the boundary at infinity.

\subsection{ Patterson-Sullivan}

\begin{theo} There exists a family of measures on $\pp \M$, $x \mapsto \nu _x, x \in \M$,  such that 
\begin{equation}\label{PS1} \nu _{\b x} = \b_\ast \nu _x,   {\textrm { for }} \b \in \G,  {\textrm { and }} 
\frac{d\nu_y}{d\nu_x} (\xi ) = e^{-Vb_{x,\xi}(y) }.
\end{equation} The family is unique if normalized by $\displaystyle \int_M \nu_x (\pp\M) \, d\vol_{{g}} (x) =1.$ 
Moreover, the measures $\nu _x$ are continuous. \end{theo}
\begin{proof}We first show the existence of such a family. Fix $x_0\in \M$. It suffices to construct  the family $\nu_{\b x_0}, \b \in \G$,  such that  
\begin{equation}\label{PS2}  {\textrm { for all }} \b  \in \G, \nu _{\b x_0} = \b_\ast \nu _{x_0}  {\textrm { and }} 
\frac{d\nu_{\b x_0}}{d\nu_{x_0}} (\xi ) = e^{-Vb_{x_0,\xi}(\b x_0)}.
\end{equation} 
Indeed, assume such a family  $\nu_{\b x_0}, \b \in \G$,  is constructed, we then set $\nu _y := e^{-Vb_{x_0,\xi}(y) } \nu _{x_0}$ for all $y \in \M$. Using the cocycle property of the Busemann function:
\[
b_{x, \xi}(\beta y)=b_{x, \beta^{-1}\xi}(y)+b_{x, \xi}(\beta x), \ \forall x, y\in \M, \xi\in \partial\M, 
\]
one can easily check that the class of measures $\{\nu _y\}$ satisfies the requirement of (\ref{PS1}).

Recall   $V =  \lim\limits_ {r \to +\infty} \frac{1}{r} \log \vol_{\wt{g}} B_{\M}(x_0, r). $ 
Set, for $s>V$, $\displaystyle d\nu ^s_{\b x_0} (y ):= \frac {e^{-sd(\b x_0,y)} \, d\vol_{\wt{g}} (y)}{\int _{\M} e^{-sd( x_0,y)} \, d\vol_{\wt{g}} (y)}.$   We have \[ \b_\ast d\nu ^s_{ x_0} (y ) \;= \; d\nu ^s_{ x_0} (\b ^{-1}y )\; = \; \frac {e^{-sd( x_0,\b^{-1} y)} \, d\vol_{\wt{g}} (y)}{\int _{\M} e^{-sd( x_0,y)} \, d\vol_{\wt{g}} (y)}\; = \; \frac {e^{-sd(\b x_0,y)} \, d\vol_{\wt{g}} (y)}{\int _{\M} e^{-sd( x_0,y)} \, d\vol_{\wt{g}} (y)}\;=\; d\nu ^s_{\b x_0} (y ).\]
Recall that $\M \cup \pp\M $ is compact and assume that $\int _{\M} e^{-sd(x_0,y)} \, d\vol_{\wt{g}} (y) \to \infty $ as $s \searrow V$. Choose $s_n \searrow V$ such that 
$\nu^{s_n}_{x_0}$ weak*  converge  to $\nu _{x_0}.$ Then, $\nu_{x_0} $ is supported by $\pp \M.$ Moreover, for any $\b \in \G$, $\nu^{s_n}_{\b x_0}$ weak* converge as well and call  $\nu _{\b x_0} := \lim_{s_n \searrow V} \nu^{s_n}_{\b x_0}$. The family $\nu _{\b x_0}, \b \in \G,$ satisfies (\ref{PS2}). Indeed, $\nu _{\b x_0} = \b_\ast \nu _{x_0} .$ Moreover, consider an open cone  $C$ based on $x_0$. We have, for any $\b \in \G$,
\[ \nu _{\b x_0} (C) = \lim\limits_{s_n \searrow V} \nu^{s_n}_{\b x_0}(C) = \lim\limits_{s_n \searrow V} \int _C e^{-s_n\left(d(\b x_0,y)-d(x_0,y)\right)}\, d\nu^{s_n}_{x_0}(y).\]
As $s_n \searrow V,$ most of the $\nu^{s_n}_{x_0}$ measure is supported by a neighbourhood of $\pp\M$ and, for $y$ close to $\xi \in \pp \M, \; d(\b x_0,y)-d(x_0,y) $ is close to $b_{x_0,\xi}(\b x_0).$ The density property follows.

If $\int _{\M} e^{-sd(x_0,y)} \, d\vol_{\wt{g}} (y)$ is bounded, use Patterson's trick (\cite[Lemma 3.1]{P}): one can find a real function $L $ on $\R_+$ such that 
\[ \lim\limits_{s \searrow V} \int _{\M} L(d(x_0,y)) e^{-sd(x_0,y)} \, d\vol_{\wt{g}} (y) =\infty  \quad {\textrm {and}}\quad  \forall a\in \R, \;  \lim\limits _{t\to +\infty } \frac{L(t + a)}{L(t)} = 1 .\]
We can then replace the previous $\nu ^s_{\b x_0} (y )$ by $\nu '^s_{\b x_0} (y ):= \frac {L(d(\b x_0,y)) e^{-sd(\b x_0,y)} \, d\vol_{\wt{g}} (y)}{\int _{\M} L(d(x_0,y)) e^{-sd( x_0,y)} \, d\vol_{\wt{g}} (y)}.$ 

The function $x \mapsto \nu_x (\pp\M) $ is $\G$-invariant and continuous; in particular, it is bounded. This implies that the measure $\nu _{x_0}$ is continuous since otherwise, there is $\xi \in \pp \M$ with $\nu _{x_0} (\{\xi \}) =a >0.$ When $\{ y_n\}_{n\in \N }\in \M $ converge to $\xi$, $\nu _{y_n} (\{\xi \}) = e^{-Vb_{x_0,\xi}(y_n)} a \to +\infty, $ a contradiction.

We will see later (Remark \ref{PSunique})  that such a family is unique, up to multiplication by a constant factor.
\end{proof}
The family $\nu _x, x \in \M, $ is called the family of {\it {Patterson-Sullivan  measures.}}

\subsection{Bowen-Margulis} 
Define, for $x \in \M, \xi, \eta \in \pp\M,$ the {\it{Gromov product }}\[ (\xi ,\eta)_x = \frac{1}{2} \lim\limits_{y \to \xi , z \to \eta} \left(d(x,y) +d(x,z ) - d(y,z)\right).\]
The Gromov product is a nonnegative number (by the triangle inequality) and because of pinched negative curvature, the Gromov product is finite; actually it is (exercise)  uniformly bounded away from the distance from $x$ to the geodesic $\g_{\eta, \xi}.$ Moreover, the Gromov product satisfies the cocycle relation 
\begin{equation}\label{cocycleGromov} (\xi ,\eta)_{x'} - (\xi ,\eta)_{x} = \frac{1}{2} (b_{x, \xi }(x') +b_{x,\eta}(x') ).\end{equation}
Let $\M^{(2)}:  = \{ (\xi, \eta ) \in \pp\M\x\pp\M, \xi \neq \eta \}$. Then, $S\M$ is identified with $\M^{(2)} \x \R $ by the {\it{Hopf coordinates:}} \[ v\mapsto (\g_v (+\infty), \g_v (-\infty ), b_v(x_0)).\]
\begin{prop} Let $\nu _x, x \in \M,$ be the family of {\it {Patterson-Sullivan }} measures. The measure $d\nu (\xi, \eta) :=  \frac{d\nu _x(\xi )\x d\nu _x(\eta)}{ e^{-2V (\xi, \eta )_x}} $ does not depend on $x$. The measure $\nu \x dt$ on $\M^{(2)} \x \R $ is $\G$-invariant  and invariant by the geodesic flow. \end{prop}
\begin{proof} The first affirmation follows directly from the cocycle relation (\ref{cocycleGromov}).  In particular, the measure $\nu $ is $\G$-invariant on $\pp\M \x \pp\M$. The measure $\nu$ is supported by $\M^{(2)}$ because $\nu _x $ is continuous. 
The actions of $\G$ and of $ \vf _s$ in Hopf coordinates are given by:
\begin{eqnarray*} \b (\xi,\eta,t) &=& (\b \xi, \b \eta, t + b_{x_0, \xi} (\b^{-1} x_0)),  {\textrm{ for }} \b \in \G, \\
\vf_s (\xi,\eta,t) &=& (\xi,\eta,t+s).\end{eqnarray*}
The invariance of $\nu \times dt$ under the actions of $\G$ and of $ \vf _s$ follows.
\end{proof}
We call {\it {Bowen-Margulis  measure }} $m_{BM} $ the unique probability measure on $SM$ such that its $\G$-invariant extension is proportional to $\nu \x dt$. It satisfies

\begin{theo}\label{BM-measure}  $h_{m_{BM}} (\vf _t ) \; = \; |t| V. $ \end{theo}
\begin{proof}({\it {Sketch}})  We follow the sketch of the proof of Theorem {\ref{Pesin}}. We have to estimate $m_{BM} \left( {\bf B}(v,\e, 1) \cap \vf _{r-1} {\bf B}(\vf _{-r+1}v,\e, 1)\right).$
Choose $\e$ small enough that this set lifts to $S\M$ into a set of the same form. In Hopf coordinates, this is, up to some constant $A$, of the form:
\[ {\bf B}(v,\e, 1) \;\asymp \; \left\{ (\xi, \eta, t):\ \xi \in C(\vf _{1/2}v, A^{\pm 1} \e), \eta  \in C(-\vf _{1/2}v, A^{\pm 1} \e),  b_v(x_0 ) \leq t \leq b_v(x_0) +1 \right\}, \]
where, for $w\in S\M$ and $0< \d <\pi , C(w, \d)$ is the cone of geodesics starting from $w$ with an angle smaller than $\d.$ Our set 
${\bf B}(v,\e, 1) \cap \vf _{r-1} {\bf B}(\vf _{-r+1}v,\e, 1)$ is \[  \left\{ (\xi, \eta, t):\ \xi \in C(\vf _{1/2}v, A^{\pm 1} \e), \eta  \in C(-\vf _{-r +3/2}v, A^{\pm 1} \e),  b_v(x_0 ) \leq t \leq b_v(x_0) +1 \right\}.\]
The  $\nu \x dt$ measure of this set  is within $A^{\pm 2}  e^{(-r +3/2)V} m_{BM} ( {\bf B}(v,\e, 1) ).$
\end{proof}

\begin{cor} $P(0) = V$ and $ m_{BM} $ is the measure of maximal entropy for the geodesic flow $\varphi_t$. In particular, $m_{BM} $ is ergodic.\end{cor}

\begin{remark}\label{PSunique} It also follows from this construction that the Patterson-Sullivan family $\nu _x$ is unique. Indeed, let  $\nu'_x$ be another Patterson-Sullivan family. One can construct as above a family $\nu'$, $d\nu' (\xi, \eta) :=  \frac{d\nu _x(\xi )\x d\nu' _x(\eta)}{ e^{-2V (\xi, \eta )_x}} .$ By the same reasoning, the mesure $\nu ' \x dt$ is proportional to an invariant probability measure with entropy $V$. It follows that $\nu '$ is proportional to $\nu$, i.e.,   $\nu '_x$ is proportional to $\nu _x$ for all $x$. \end{remark}

\subsection{Burger-Roblin}

Define a measure $\wt m_{BR}$ on $\M\x \pp \M$ by setting, for all continuous function $F$ with compact support on $S\M$,
\begin{equation}\label{BRmeasure} \int F (x,\xi )\, d\wt m_{BR} =   \int _{\M} \left(\int _{\pp \M} F(x, \xi) \, d\nu_x (\xi) \right)  d\vol_{\wt{g}} (x).\end{equation}
It follows from the definition that the measure $\wt m_{BR}$ is $\G$-invariant. Call $m_{BR}$ the induced 
 measure on $SM$; by our normalization,  we have $m_{BR}(SM) = 1.$ The measure $m_{BR}$ is called the 
{\it {Burger-Roblin measure}.} Many of its properties follow from 

\begin{theo}  For any vector field $Z$ on $SM$ such that  $Z(v)$ is tangent to $W^s(v)$ for all $v \in SM$, we have
 \begin{equation}\label{int.byparts} 
\int_{SM} {\rm{Div}}^s Z (v) + V <Z(v), \ov X (v)> \, dm_{BR} (v) \;= \;0 .\end{equation}
\end{theo}
\begin{proof}  Using a partition of unity, we may assume that $Z$ has compact support inside a flow-box for the foliation. 
Choosing a reference point $x_0$, we can write  $dm_{BR} (x,\xi) = e^{-Vb_{x_0,\xi}(y)} d\nu_{x_0} (y) d\vol_{\wt{g}} (y). $ Since 
$Z$ has compact support on each local stable leaf $W^s_{loc}(x,\xi )$, we have \[\int _{W^s_{loc}(x,\xi )} {\textrm {Div}}^s_y\left(  e^{-Vb_{x_0,\xi}(y)} Z (y,\xi ) \right)\Big|_{y=z}\, d\vol_{\wt{g}} (z) =0\]
for all $(x,\xi) \in S\M$.  Then, (\ref{int.byparts}) follows by developing  \[ {\textrm {Div}}^s_y\left(  e^{-Vb_{x_0,\xi}(y)} Z (y,\xi ) \right)\Big|_{y=z}= \left({\textrm {Div}}^s_y Z (y,\xi )\Big|_{y=z} + V <Z(z,\xi ), \ov X (z,\xi )> \right) 
 e^{-Vb_{x_0,\xi}(z)} .\]
\end{proof}

\begin{cor}   $\int B \, dm_{BR} = V.$ \end{cor} \begin{proof} Apply (\ref{int.byparts}) to $Z= \ov X.$\end{proof} 

\begin{cor} \label{BRsymmetric} The operator $\D^s +V\ov X$ is symmetric for $m_{BR}$: for $F_1,F_2 \in C^\infty (SM),$ the set of smooth functions on $SM$, \[\int _{SM} F_1 (\D^s +V\ov X) F_2 \,  dm_{BR} =\int _{SM} F_2 (\D^s +V\ov X) F_1 \, dm_{BR}.\]
Hence,  $m_{BR}$ is also  {\it {stationary }} for the operator $\D^s +V\ov X$, i.e.,  for all $F \in C^\infty (SM)$,   $\int _{SM}  (\D^s +V\ov X) F \  dm_{BR} = 0.$ \end{cor}
\begin{proof} Apply (\ref{int.byparts}) to $Z=  F_1 \nabla^s F_2$ to get
\[\int _{SM} F_1 (\D^s +V\ov X) F_2 \,  dm_{BR} \; = \; -\int _{SM} <\nabla ^s F_1, \nabla ^s F_2> \, dm_{BR}.\]
The Right Hand Side is invariant when switching $F_1$ and $F_2.$ 
\end{proof}

\begin{cor} \label{BRstrongstable} The measure $m_{BR}$ is symmetric for the Laplacian $\D^{ss} $ along the strong stable foliation $\W^{ss}$: for $F_1,F_2  \in C^\infty (SM),$
\[ \int _{SM} F_1 \D^{ss} F_2 \,  dm_{BR} =  \int _{SM} F_2 \D^{ss} F_1 \, dm_{BR}.\]
So, $m_{BR}$ is also {\it {stationary }} for the operator $\D^{ss}$, i.e.,  for all $F \in C^\infty (SM)$,  $\int _{SM}  \D^{ss} F \  dm_{BR} = 0.$
\end{cor}
\begin{proof} Apply (\ref{int.byparts}) to $Z= F_1\frac{d}{dt}F_2\circ \vf _t \big|_{t=0}  \ov X$ to obtain that 
\[ \int_{SM} F_1\left(\frac{d^2}{dt^2}F_2\circ \vf _t \big|_{t=0} - B\frac{d}{dt}F_2\circ \vf _t \big|_{t=0} + V\frac{d}{dt}F_2\circ \vf _t \big|_{t=0}\right) dm_{BR}=- \int_{SM} \ov XF_1 \ov XF_2 \, dm_{BR}. \] 
Recall that  in horospherical coordinates, $\D^s$ can be written as
\[ \D^s F = \frac{d^2}{dt^2}F\circ \vf _t \big|_{t=0} - B\frac{d}{dt}F\circ \vf _t \big|_{t=0} +\D^{ss} F.\]
Replacing in the formula above, we get that 
\[ - \int _{SM} F_1 \D^{ss} F_2  \, dm_{BR}+\int_{SM} F_1 (\D^s +V\ov X) F_2 \, dm_{BR} \;=\; - \int_{SM} \ov XF_1 \ov XF_2 \, dm_{BR}. \] 
 The conclusion follows from Corollary \ref{BRsymmetric}.\end{proof}

\begin{remark}\label{BRergodic} We observe that  $m_{BR}$ is ergodic. Indeed, strong stable manifolds have polynomial volume  growth\footnote{There are constants $C, k$ such that the volume of the balls of radius $r$ for the induced metric on strong stable manifolds is bounded by  $Cr^k$.}, 
so a symmetric measure for the Laplacian $\D^{ss} $ along the strong stable foliation $\W^{ss}$  is given locally  by the product of the Lebesgue measure along the $W^{ss}$ leaves and some family of measures on the transversals (Kaimanovich, \cite {Ka2}).  This family has to be {\it{invariant }} under the holonomy map of the  $W^{ss}$ leaves. By Bowen-Marcus (\cite{BM}), there exists only one holonomy-invariant family on the transversals to the $\W^{ss}$ foliation, up to  a multiplication by a constant factor.
\end{remark}

\begin{remark} The family of measures in this section has a long history. The invariant measures for the $\W^{ss}$ foliation were first constructed by Margulis (\cite {Ms}) and used to construct the invariant measure $m_{BM}$. Margulis' construction (in the strong unstable case) amounts to taking  the limit of the normalized Lebesgue measure on $\vf _T S_xM$ (see also Knieper (\cite{Kn})).  Margulis did not state that the measure $m_{BM}$ has maximal entropy, and the measure of maximal entropy was constructed by Bowen (cf. Bowen (\cite{B2}), Bowen-Ruelle (\cite{BR})) as the limit as $T\to + \infty $ of equidistributed measures on closed geodesics  of length smaller than $T$. Bowen also showed that the measure of maximal entropy is unique, so that the two constructions give the same measure $m_{BM}$. Independently, Patterson (\cite{P}) constructed the measures $\nu _x$ in the case of hyperbolic surfaces, not necessarily compact;  Sullivan (\cite{Su})  extended the construction to a general hyperbolic space, observed that it  is, up to normalization,  the Hausdorff measure on the limit set of the discrete group in its Hausdorff dimension  for the  angle  metric,    that it is also the conformal measure for the action of the group on its limit set and moreover, the exit measure of the Brownian motion  with suitable drift.   He also   made its connection with the measure of maximal entropy (in the constant curvature case).  Hamenst\"adt (\cite{H1}) connected $m_{BM}$ with the Patterson-Sullivan construction and then many authors extended the Patterson-Sullivan construction to many circumstances (see Paulin-Pollicott-Schapira \cite{PPS} for a detailed recent survey).  Again in the  hyperbolic geometrically finite case, Burger (\cite{Bu}) considered $m_{BR}$ as the measure invariant by the horocycle action; finally, Roblin (\cite{Ro}) considered the general case of a group acting discretely on a $CAT(-1)$ space. What is remarkable is that in all these constructions, these measures were introduced as  tools, and not, like here, as  objects interesting in their own right. A posteriori, their interest comes from all these applications.
\end{remark}

\

\section{A family of stable diffusions; probabilistic rigidity}
Recall (Corollary \ref{BRsymmetric}) that  the Burger-Roblin measure $m_{BR}$ is a stationary measure for $\D^s +V\ov X$. In this section, we study the stationary measures for $\D^s +\rho \ov X, \rho <V,$  characterize them in analogy to $m_{BR}$, and state a rigidity result concerning these measures.

\subsection{Foliated diffusions}A differential operator $\LL$ on $SM$ is called {\it{subordinated to the stable foliation $\W^s$}} if, for any $F\in C^{\infty}(SM)$,  $\LL F (v) $ depends only on the values of $F$ along $W^s (v).$ It is given by a $\Gamma$-equivariant family $\LL_\xi $ on $\M \x \{ \xi \}.$ A  probability measure $m$ is called {\it{stationary}} for $\LL$  if, for all $F\in C^{\infty}(SM)$, 
\[ \int \LL F(v)\, dm(v) \; = \; 0.\] 
 \begin{theo}[Garnett, \cite{Ga}]\label{Ga-theo} Assume $\LL $ is an  operator which is subordinated to $\W^s$,   has continuous coefficients, and is elliptic on $W^s$ leaves. Then, the  set of $\LL$ stationary probability measures is a non-empty convex compact set. Extremal points are called ergodic. \end{theo}

We will consider the operators $\LL^\rho:=\Delta^s+\rho\ov{X}$ for $\rho\in \Bbb R$. Clearly, each  $\LL^\rho$ is subordinated to $\W^s$  and for all $F\in C^{\infty}(SM)$, 
\[ \LL _\xi ^\rho  F (x,\xi ) = \D_y^s F(y, \xi)|_{y=x} +\rho <\ov X, \nabla _y^s F(y, \xi ) |_{y=x}>_{x,\xi}.\]
For a fixed $\xi,$ $  \LL _\xi ^\rho $ is elliptic on $\M$ and Markovian ($\LL _\xi ^\rho 1 = 0$). Hence, by Theorem \ref{Ga-theo},  there is always some $\LL^{\rho}$ stationary  measure. Let $m_{\rho}$ be a $\LL^{\rho}$ stationary  measure. Then, locally  (\cite{Ga}), on a local flow-box of the lamination the measure $m_{\rho}$ has conditional measures along the leaves that are absolutely continuous with respect to Lebesgue, and the  density $K^{\rho}$ satisfies $\LL^{\rho \ast} K^{\rho} = 0 ,$ where $\LL^{\rho \ast}$ is the formal adjoint of $\LL^\rho$ with respect to Lebesgue measure on the leaf, i.e., 
\begin{equation}\label{equ-dual}\LL^{\rho \ast} F \; = \; \D^s F - \rho {\textrm {Div}}^s (F \ov X). \end{equation}
 Globally,  there exists a $\Gamma$-equivariant  family of measures $\nu^\rho_x $ such that the $\G$-invariant extension $\wt m_\rho $ of $m_\rho $ is given by a formula analogous to (\ref{BRmeasure}): \[  \int F (x,\xi )\; d\wt m_\rho\: = \:  \int _{\M} \left(\int _{\pp \M} F(x, \xi) \, d\nu^\rho_x (\xi) \right)  d\vol _{\wt g}(x).\] 
 Indeed, choose a transversal to the foliation $\W^s$, say the sphere $S_{x_0} M$ and write $SM$ as $M_0 \x S_{x_0} M.$   A stationary measure $m_\rho$ is given by an integral for   some measure $d\nu (\xi) $ of measures of the form  $\mathsf{K}^\rho(x, \xi )\: d\vol_{g}(x)$,  where $\vol_{g}$ is the volume on $M_0$.  We can arrange that  $\mathsf{K}^\rho(x_0, \xi ) = 1,\ \nu $-a.e.. For a lift $\wt x_0= :x$, set $\nu ^\rho_x = (\pi _{x })_\ast \nu$. The family $\nu ^\rho_{\b x}, \b \in \G, $ is $\Gamma$-equivariant by construction. Starting from a different point $y_0 \in M_0$, the same construction gives  a $\Gamma$-equivariant family $\nu ^\rho_{\b y}, \b \in \G, $ for the lifts $y$ of $y_0$. By construction also, 
 \[ \frac{d\nu^\rho_y}{d\nu^\rho_x} (\xi )\: =\: \frac{\mathsf{K}^\rho(y,\xi )}{\mathsf{K}^\rho(x,\xi )}.\]
 The same proof as for the relation (\ref{int.byparts}) yields, for any  vector field $Z$ on $SM$ such that  $Z(v)$ is tangent to $W^s(v)$ for all $v \in SM$, 
\begin{equation}\label{int.byparts2} 
\int_{SM} {\textrm {Div}}^s Z   + <Z, \nabla_y^s \log \mathsf{K}^\rho (y,\xi)\big|_{y=x} >\, dm_{\rho} (v) =0.\end{equation}

  For each $\LL^{\rho}$, there is a {\it{diffusion}}, i.e., a $\Gamma$-equivariant  family of probability measures $\wt \P^\rho _{x,\xi}   $ on $C(\R_+, S\M)$ such that 
 $t \mapsto \wt \om (t) $ is a Markov process with generator $\LL ^\rho _\xi $,
$\wt\P ^\rho _{x,\xi} $-a.s. $\wt \om(0) = (x,\xi )$ and $\wt \om( t) \in \M\x \{\xi \},\;  \forall t>0.$
The distribution of $\wt \om(t) $ under $ \wt\P ^\rho _{x,\xi}$ is $p^\rho_\xi (t,x,y)\: d\vol_{\wt{g}}(y)\: \d_\xi (\eta),$  where $p_\xi^\rho (t,x,y) $ is the fundamental solution of the  equation $\frac{\pp F}{\pp t} = \LL _\xi ^\rho F $.   The quotient $\P^\rho_v$ defines a Markov process on $SM$ such that for all $t \geq 0$, $\om (t) \in W^s(\om (0)).$  For any $\LL^{\rho}$ stationary measure  $m_{\rho}$,  the probability measure $\P_{m_{\rho}}^\rho := \int \P_v^{\rho} \,dm_{\rho}(v) $ is invariant under the  shift on $C(\R_+, SM)$ (cf. \cite{Ga, H2}).  If the measure $m_{\rho}$ is an extremal point of the set of stationary measures for $\LL^\rho$, then the probability measure $\P_{m_{\rho}}^\rho$ is invariant  ergodic under the shift on $C(\R_+, SM).$

\begin{prop}\label{cocycle-ell} Let  $m_\rho$ be a stationary ergodic measure for $\LL^\rho$. Then, for $\P_{m_{\rho}}^\rho $ a.e. $\om$ and  any lift $\wt{\om}$ of $\om$ to $S\M$, 
\begin{equation}\label{ell} \lim\limits_{t \to +\infty } \frac{1}{t} b_{\wt \om (0)} (\wt \om (t))\;=\;-\rho + \int B \, dm_\rho\;=:\;\ell_{\rho} (m_\rho ).\end{equation}
In particular, for $\rho = V, m_\rho = m_{BR} $, we have $\ell_{V} (m_{BR}) = V- \int B \, dm_{BR} = 0.$ \end{prop}
By Remark \ref{BRergodic}, the  measure $m_{BR}$ is ergodic.

 \begin{proof} Let  $\s_t, t\in \Bbb R_+$,  be the shift transformation on  $C(\R_+, S\M).$ For any $\wt{\om}\in C(\R_+, S\M)$, $t, s\in \R_+$,  
$ b_{\wt \om (0)} (\wt \om (t+s)) = b_{\wt \om (0)} (\wt \om (t)) +  b_{\s_t \wt \om (0)} (\s_t \wt \om (s)).$ By $\G$-equivariance, $ b_{\wt \om (0)} (\wt \om (t))$ takes the same value for all $\wt{\om}$ with the same projection in  $C(\R_+, SM)$  and defines an additive functional on  $C(\R_+, SM).$  Moreover, $\sup_{0\leq t\leq 1}  b_{\wt \om (0)} (\wt \om (t))  \leq \sup_{0\leq t\leq 1}  d(\wt \om (0), \wt \om (t)),$
 so that the convergence  in (\ref{ell}) holds  $\P_{m_\rho}^\rho$-a.e. and in $L^1(\P_{m_\rho}^\rho).$   By ergodicity of the process and additivity of the functional  $ b_{\wt \om (0)} (\wt \om (t))$, the limit is $ \frac{1}{t} \E_{m_\rho}^\rho  \left(b_{\wt \om (0)} (\wt \om (t))\right) $, for all $t>0$. 
 In particular, 
 \begin{align*}
 \lim\limits_{t \to +\infty } \frac{1}{t} b_{\wt \om (0)} (\wt \om (t))\; =\;& \lim\limits_{t \to 0^+} \frac{1}{t} \E_{m_\rho}^\rho \left(b_{\wt \om (0)} (\wt \om (t))\right)\\
 \;=\;& \int_{SM} \D_y^s b_{x,\xi} (y)\big|_{y=x} +\rho <\ov X, \nabla _y^s b_{x,\xi} (y)\big|_{y=x} >_{x,\xi}  \, dm_\rho (x,\xi).
 \end{align*}
 Equation (\ref{ell}) follows. 
 \end{proof}

Following Ancona (\cite{Anc}) and Hamenst\"adt (\cite{H2}), we call our operator  $\LL^{\rho}$ {\it {weakly coercive}} if there is some $\e >0 $ such that for all $\xi \in \pp \M,$ there exists a positive superharmonic function for the operator  $\LL _\xi ^\rho + \e $ (i.e.  a positive $F$ such that  $\LL _\xi ^\rho F + \e F \leq 0$).   As a corollary of Proposition \ref{cocycle-ell}, we see that  if  $m_{\rho}$ is a  $\LL^{\rho}$ stationary measure with $\ell_{\rho}(m_{\rho})>0$, then for $\wt{m}_{\rho}$ almost all $\wt{\om}(0)$ and $\wt\P ^\rho _{x,\xi}$ almost all $\wt{\om}$, $\wt{\om}(+\infty)=\lim_{t\to +\infty} \wt{\om}(t)\in (\partial\M \setminus \{\xi\}) \times \{\xi\}$.  This, together with the  negative curvature  and the cocompact assumption of the underlying space, implies that 

\begin{cor}{\rm(\cite[Corollary 3.10]{H2})}\label{cor-weakly-c} Assume the operator $\LL^\rho$ is such that there exists some $\LL^{\rho}$  stationary ergodic   measure  $m_\rho$  with $\ell_{\rho}(m_{\rho})>0$. Then,  $\LL^\rho$ is weakly coercive. 
\end{cor}

\subsection{Stable diffusions}  For a weakly coercive $\LL^{\rho}$, we want to understand more about its diffusions. 
Hamenst\"adt developed in \cite{H2} many tools for the study of the foliated diffusions subordinated to the stable foliation $\W^s$, using dynamics and thermodynamical formalism. We review in this subsection her results when applied  for our $\LL^{\rho}.$

For each $\LL^{\rho}$, $\rho\in \Bbb R$, recall  that  $p_\xi^\rho (t,x,y) $ is the fundamental solution of the  equation $\frac{\pp F}{\pp t} = \LL _\xi ^\rho F $. We write $G_\xi ^\rho (x,y)$ for the Green function of $\LL^{\rho}$: for $x,y \in \M,$
\[G_\xi ^\rho (x,y) \; := \; \int _0^\infty p_\xi^\rho (t,x,y) \, dt.\] 
For weakly coercive operators on a pinched negatively curved simply connected manifold, Ancona's Martin boundary theory  (\cite{Anc}) shows the following
\begin{theo} [\cite{Anc}] Assume that the operator $\D^s + \rho \ov X$ is weakly coercive and recall that the sectional curvature of $\M$ is between two constants  $-a^2 $ and $-b^2.$ There exists a constant $C$ such that for any $\xi \in \pp \M$, any three points $x,y,z$ in that order on the same geodesic in $\M$ and such that $d(x,y), d(y,z) \geq 1,$ we have:
\begin{equation}\label{Ancona}  C^{-1} G_\xi ^\rho (x,y) G_\xi ^\rho (y,z)\; \leq \; G_\xi ^\rho (y,z)\; \leq \; C G_\xi ^\rho (x,y)G_\xi ^\rho (y,z). \end{equation}
\end{theo}

(In particular, by Corollary \ref{cor-weakly-c}, the inequality (\ref{Ancona}) holds for $\rho$ such that there is an ergodic $\LL^{\rho}$ stationary measure $m_{\rho}$ with $\ell_{\rho}(m_{\rho})>0$.) 
 
Ancona (\cite{Anc}) deduces from (\ref{Ancona}) that the {\it {Martin boundary }} of each weakly coercive operator $\LL_\xi ^\rho $ is the geometric boundary $\pp \M.$ Namely, for any $x,y \in \M, \xi , \eta \in \pp\M,$ there exists a function $K_{\xi, \eta} ^\rho (x,y) $ such that 
\[ \lim\limits _{z \to \eta} \frac{G_\xi ^\rho (y,z) }{G_\xi ^\rho (x,z )} \; = \; K_{\xi, \eta} ^\rho (x,y) .\]
 The function $K_{\xi, \eta} ^\rho (x,y)$ is $\LL_{\xi} ^\rho $-harmonic and therefore smooth in $x$ and $y$. Moreover, the functions $(x,\eta) \mapsto K_{\xi, \eta} ^\rho (x,y),\;  (x,\eta)\mapsto \nabla _y K_{\xi, \eta} ^\rho (x,y) \big|_{y = x} $ are H\"older continuous (cf. \cite{H2}, Appendix B).  By uniformity of the constant $C$ in (\ref{Ancona}),  the functions $(x,\xi) \mapsto K_{\xi, \eta} ^\rho (x,y),$    $\xi \mapsto \nabla _y K_{\xi, \eta} ^\rho (x,y) \big|_{y = x} $ are continuous into the space of H\"older continuous functions on $SM$ (see e.g. \cite{LL}).

 Let $\LL ^{\rho \ast}$ be the leafwise formal adjoint of $\LL ^\rho$ (see (\ref{equ-dual})).  Then,  $\LL ^{\rho \ast}$ is subordinated to $\W^s$ and the corresponding Green function $ G_\xi ^{\rho \ast} (x,y)$ is given by $ G_\xi ^{\rho \ast} (x,y) = G_\xi ^{\rho} (y,x).$ In particular, the Green function $G_\xi ^{\rho \ast} (x,y)$ satisfies (\ref{Ancona}) as well and we find, for $\xi, \eta \in \pp\M, x,y \in \M$,   the Martin kernel $K_{\xi , \eta} ^{\rho \ast} (x,y)$ given by:
 \[ K_{\xi , \eta} ^{\rho \ast} (x,y)\; = \; \lim\limits_{z \to \eta} \frac{G_\xi ^{\rho \ast} (y,z)}{G_\xi ^{\rho \ast} (x,z)} \;= \; \lim\limits_{z \to \eta} \frac{G_\xi ^{\rho} (z,y)}{G_\xi ^{\rho} (z,x)}.\]
 Again, the function  $K_{\xi, \eta} ^{\rho\ast} (x,y)$  is $\LL_{\xi} ^\rho $-harmonic and therefore smooth in $x$ and $y$. Moreover, the functions $(x,\eta) \mapsto K_{\xi, \eta} ^{\rho\ast} (x,y),\;  (x,\eta )\mapsto \nabla _y K_{\xi, \eta} ^{\rho\ast} (x,y) \big|_{y = x} $ are H\"older continuous and  the functions  $(x,\xi )  \mapsto K_{\xi, \eta} ^{\rho\ast} (x,y),$   $\xi \mapsto \nabla _y K_{\xi, \eta} ^{\rho\ast} (x,y) \big|_{y = x} $ are continuous into the space of H\"older continuous functions on $SM$. Observe also that the relation (\ref{Ancona}) is satisfied also by the resolvent  $G_\xi ^{\l, \rho \ast} (x,y) \; := \; \int _0^\infty e^{-\l t} p_\xi^\rho (t,y,x) \, dt $, uniformly for $\l  >0$ close to 0 and for $\xi \in \pp\M$, so that we also have:
 \begin{equation}\label{uniformAncona}  K_{\xi, \eta} ^{\rho\ast} (x,y) = \lim\limits _{z \to \eta, \l \to 0^+ } \frac{G_\xi ^{\l, \rho \ast} (y,z)}{G_\xi ^{\l, \rho \ast} (x,z)}. \end{equation}
 
We can use  the function $K_{\xi, \eta} ^{\rho,\ast} (x,y) $ to express the function $\mathsf{K}^\rho $ in (\ref{int.byparts2}).
 \begin{prop} \label{density} Assume $\ell_{\rho} (m_\rho) >0$ and $m_\rho $ is ergodic. Then, the corresponding $\mathsf{K}^\rho $ in (\ref{int.byparts2}) is given by   $ \frac{\mathsf{K}^\rho (y,\xi)}{\mathsf{K}^\rho (x,\xi)} = K_{\xi, \xi } ^{\rho\ast} (x,y). $ 
 \end{prop}
\begin{proof} Let $\nu _x^\rho$ be  the family such that $d \wt m_\rho (x,\xi )  = d\vol_{\wt{g}} (x) d\nu ^\rho_x (\xi ).$ 
 For  $F $ a bounded measurable  function on $SM$,  set $\wt F$ for the $\G$-periodic function on $\M \x \pp \M$ extending $F$. Since $m_\rho$ is ergodic, we have, for $m_\rho$-a.e. $(x,\xi)$,
 \[ \int _{SM} F\, dm_\rho 
\;=\;  \lim\limits _{\l  \to 0^+} \l \int_0^\infty e^{-\l t}\left(\int p^\rho_\xi (t,x,y) \wt F (y,\xi )\, d\vol _{\wt{g}}(y)   \right) \, dt. \]
  The inner integral can be written  
\[ \sum _{\b \in \G} \int p^\rho_\xi (t,x,\b y) \wt F (\b y,\xi )\,  d\vol _{g}(y)  = \sum _{\b \in \G} \int p^\rho_{\b^{-1} \xi} (t,\b^{-1} x,y) \wt F (y, \b^{-1}\xi )\,  d\vol _{g}(y),\] 
where $\vol_{g}$ is the restriction of  $\vol_{\wt{g}}$  on the fundamental domain $M_0$,   so that we have
 \[ \int _{SM} F\, dm_\rho = \lim\limits _{\l  \to 0^+}  \sum _{\b \in \G} \l \int G_{\b^{-1}\xi}^{\l, \rho \ast} (y,\b^{-1}x)  F (y, \b^{-1}\xi)\, d\vol _{g}(y) .\]
By Harnack inequality, all ratios $\frac{ G_{\b^{-1}\xi}^{\l, \rho \ast} (y,\b^{-1}x)} { G_{\b^{-1}\xi}^{\l, \rho \ast} (z,\b^{-1}x)} $ for $y,z \in M_0$ are of the same order as soon as $ d(\b^{-1} x, M_0) \geq 1.$  Choose an open $A \subset \pp\M$ disjoint from $\{\xi \}$. If, for $\b$ large enough, $\b^{-1} \xi \in A,$ then  $\b^{-1} x $ is close to $A$.
  Then, by (\ref{uniformAncona}) and Harnack inequality, given $\e >0$, for all $x \in M_0, \xi \in \pp \M$, for all $\b \in \G$ so that $\b^{-1} x $ is close enough to $\b^{-1} \xi $,  $y'$ close enough to $y$, $z'$ close enough to $z$, 
\[  \frac{G_{\b^{-1}\xi} ^{\l, \rho \ast} (y',\b^{-1}x)}{G_{\b^{-1}\xi }^{\l, \rho \ast} (z',\b^{-1}x)} \; \sim^{1+\e}\;  K_{\b^{-1}\xi , \b^{-1}\xi } ^{\rho,\ast} (z,y),\]
where,  for $a, b\in \Bbb R$,  $a \sim^{1+\e} b$ means  $(1+\e)^{-1}b\leq a\leq (1+\e)b$. 
Consider as  functions $F_y, F_z $ the indicator of $\U_y \x A, \U_z \x A$, where $\U_y,\U_z$ are respectively  small neighborhoods of $y, z$. Then, 
\begin{eqnarray*} \int _{SM} F_y\, dm_\rho = \int _{U_y} \nu ^\rho _{y'}(A) \, d\vol _g(y')& = & \lim\limits _{\l  \to 0^+}  \sum _{\b \in \G, \b^{-1} \xi \in A} \l \int _{U_y} G_{\b^{-1}\xi}^{\l, \rho \ast} (y',\b^{-1}x)\,  d\vol _{g}(y') \\ \int _{SM} F_z\,  dm_\rho = \int _{U_z} \nu ^\rho _{z'}(A) \, d\vol _g(z')& = & \lim\limits _{\l  \to 0^+}  \sum _{\b \in \G, \b^{-1} \xi \in A} \l \int _{U_z} G_{\b^{-1}\xi}^{\l, \rho \ast} (z',\b^{-1}x)\,  d\vol _{g}(z').\end{eqnarray*}
As $\l \to 0^+,$ 
the  $\b$'s  involved in the sums are such that the distance $d(y, \b^{-1} x) , d(z, \b^{-1} x)$ is larger and larger.
It follows that, for  $\nu _z^\rho$-a.e. $\eta$,   \[ \frac{d \nu _y^\rho }{d \nu _z^\rho }(\eta )\; = \; K_{\eta,\eta} ^{\rho,\ast} (z,y).\]
    \end{proof}
    
\begin{cor}\label{ergodicity} Assume  $\ell _\rho(m_{\rho}) >0$ for some ergodic $\LL^\rho $ stationary  measure $m_\rho $. Then, $m_\rho $ is the only $\LL ^\rho $ stationary probability measure. \end{cor}
\begin{proof} By Proposition \ref{density}, any ergodic $\LL^\rho $ stationary measure is described by a $\Gamma$-equivariant family of measures at the  boundary  $\nu _x$ that satisfies 
\[ \frac{d \nu _y }{d \nu _z}(\eta )\; = \; K_{\eta,\eta} ^{\rho,\ast} (z,y).\]  Since the cocycle depends  H\"older continuously on $\eta$, there is a unique equivariant family with that property (see,  e.g.,  \cite[Th\'eor\`eme 1.d]{L1}, \cite[Corollary 5.12]{PPS}). \end{proof}

 \subsection{ Stochastic  entropy and rigidity}

Let $m_\rho$ be an ergodic  $\LL^{\rho}$ stationary measure, and assume that  $\ell_{\rho} (m_\rho) >0.$   The following theorems are the counterpart of the more familiar random walks properties in our setting.

\begin{theo}[Kaimanovich,  \cite{K1}] Let $m_\rho$ be an ergodic $\LL^{\rho}$ stationary measure, and assume that   $\ell_{\rho} (m_\rho) >0.$   For $\P_{m_\rho} $-a.e. $\om \in C(\R_+, SM)$, the following limits exist
\begin{eqnarray*} h_{\rho}(m_\rho) &=& \lim\limits_{t \to +\infty } -\frac{1}{t} \log p_\xi^\rho (t, \wt \om(0), \wt \om (t) ) \\
&=& \lim\limits_{t \to +\infty } -\frac{1}{t} \log G_\xi ^\rho(\wt \om (0), \wt \om (t)), \end{eqnarray*}
where $\wt \om (t), t\geq 0, $  is a lift of $\om$ to $S\M.$  Moreover, 
\[h_{\rho}(m_\rho)= \int_{SM}\left( \|\nabla ^s \log  {\mathsf{K}}^\rho(x, \xi) \|^2 - \rho B(x,\xi) \right) \, dm_{\rho}.\]
\end{theo}

\begin{proof}The first part is proven in details in \cite{LS2}, Proposition 2.4. For the final formula, we follow \cite{LS2}, Erratum. Since the notations are not exactly the same, for the sake of clarity, we give the  main ideas of the proof. We firstly  claim is that,  since $\ell_{\rho} (m_\rho) >0,$  for  $\P_{m_\rho}$-a.e. $\om \in C(\R_+, SM),$
\[
\limsup\limits_{t\to +\infty}\left|\log G_{\xi}^\rho (\wt {\om}(0), \wt {\om}(t))-\log
K^{\rho\ast}_{\xi,\xi}(\wt {\om}(0), \wt {\om}(t))\right|<+\infty.
\]
 Indeed, let $z_t$ be the point on the geodesic ray $\g_{\wt{\om}(t), \xi}$ 
closest to $x$. Then, as $t \to +\infty $,
\begin{equation*}
 G_{\xi}^\rho (\wt {\om}(0), \wt {\om}(t))\asymp G_{\xi}^\rho (z_t, \wt {\om}(t)) \asymp \frac{G_{\xi}^\rho (y, \wt {\om}(t)) }{G_{\xi}^\rho (y, z_t)}
\end{equation*}
for all $y$ on the geodesic going from $\wt {\om} (t)$ to $\xi$ with $d(y, \wt \om(t) )\geq d(y, z_t) +1$,
where $\asymp$ means up to some multiplicative constant independent of $t$. The
first $\asymp$  comes from Harnack inequality using the fact
that $\sup_t d(x, z_t)$ is finite $\P_{m_{\rho}}$-almost  everywhere. (Since $\ell_{\rho}(m_\rho) >0$, for $\P_{m_\rho} $-a.e. $\om \in C(\R_+, SM)$, $\eta=\lim_{t\rightarrow +\infty}\wt \om(t)$ differs from $\xi$ and  $d(x, z_t)$, as $t\rightarrow +\infty$,  converge to  the distance between $x$ and  $\g_{\xi, \eta}$.)
The second $\asymp$  comes from Ancona inequality (\ref{Ancona}).
Replace $\frac{G_{\xi}^\rho (y, \wt {\om}(t)) }{G_{\xi}^\rho (y, z_t)}$
 by its limit
as $y \to \xi $, which is $K^{\rho\ast}_{\xi,\xi}(z_t, \wt {\om}(t)) $ by (\ref{uniformAncona}),
which is itself $\asymp K^{\rho\ast}_{\xi,\xi}(\wt {\om}(0), \wt {\om}(t))$ by
Harnack inequality again.
It follows that,  for  $\P_{m_\rho}^{\rho} $-a.e. $\om \in C(\R_+, SM),$ 
\[ h_\rho(m_{\rho}) \; = \;  \lim\limits_{t \to \infty } -\frac{1}{t} \log K^{\rho\ast}_{\xi,\xi}(\wt {\om}(0), \wt {\om}(t)). \] 
By Harnack inequality, there is a constant $C$ such that $|\log K^{\rho\ast}_{\xi,\xi}(\wt {\om}(0), \wt {\om}(t)) | \leq C d(\wt \om (0), \wt \om (t)) . $ Since $ \log K^{\rho\ast}_{\xi,\xi}(\wt {\om}(0), \wt {\om}(t))$ is additive along the trajectories, and $\P_{m_\rho}^{\rho}$ is shift ergodic, the limit reduces to 
\begin{eqnarray*}  h_\rho(m_{\rho}) &=&  \lim\limits_{t \to 0^+ }-\frac{1}{t} \E_{m_\rho} \log K^{\rho\ast}_{\xi,\xi}(\wt {\om}(0), \wt {\om}(t))\\
&=& -\int_{SM} \left( \D_y^s \log K^{\rho\ast}_{\xi,\xi}(x,y)\big|_{y=x} +\rho <\ov X, \nabla _y^s \log K^{\rho\ast}_{\xi,\xi}(x,y)\big|_{y=x} >_{x,\xi} \right) \, dm_\rho (x,\xi)\\
&=& - \int_{SM} \left( \D^s \log \mathsf{K}^\rho (x,\xi ) +\rho <\ov X, \nabla ^s \log \mathsf{K}^{\rho}> (x,\xi )  \right) \, dm_\rho (x,\xi) ,\end{eqnarray*}
where we used Proposition \ref{density} to replace  $\nabla _y^s \log K^{\rho\ast}_{\xi,\xi}(x,y)\big|_{y=x} $ by  $ \nabla ^s \log \mathsf{K}^{\rho} (x,\xi ).$    Finally, we use (\ref{int.byparts2}) applied to $Z =  \nabla ^s \log \mathsf{K}^{\rho} (x,\xi )$ to write \[ - \int_{SM} \D^s \log \mathsf{K}^\rho (x,\xi ) \, dm_\rho (x,\xi) = \int_{SM} \|  \nabla ^s \log \mathsf{K}^{\rho} (x,\xi ) \|^2 \, dm_\rho (x,\xi)\] and applied to $Z = \ov X$ to write 
\begin{equation}\label{basic}  \int B \, dm_{\rho}  = \int  <\ov X,   \nabla ^s\log \mathsf{K}^\rho >\, dm_{\rho}.\end{equation} 
The formula for the entropy follows.
\end{proof}

\begin{theo}[Guivarc'h, \cite{Gu}]\label{guivarc'h} Assume that $\ell_{\rho} (m_\rho) >0.$  Then,  $h_{\rho}(m_\rho )\; \leq \; \ell_{\rho}(m _\rho) V.$ \end{theo}
 \begin{proof} Fix  $(x,\xi ) \in S\M $  such that $\frac{1}{t} b_{x,\xi} (\wt \om (t) ) \to \ell_{\rho} (m_\rho)$  and $ -\frac{1}{t} \log p_\xi^\rho (t, \wt \om(0), \wt \om (t) ) \to h_{\rho}(m_\rho )$, $\wt{\P}^\rho _{x,\xi} $-a.e. as $t \to +\infty$. There is a constant $\wt C$  depending only on the curvature bounds such that one can find a partition $\mathcal A = \{ A_k, k\in \N \} $  of $\M$ such that the sets $A_k$ have diameter  at most  $\wt C$ and inner diameter at least $1$. Set for $k \in \N, t >0$, $q_k^\rho (t) := \wt{\P}^\rho _{x,\xi} [ \wt \om (t) \in A_k ].$ The family $\{q_k^\rho (t), k\in \N \}$ is a probability on $\N$ with the property that, with high probability, $q_k^\rho (t) \leq e^{-t (h_{\rho}(m_\rho ) -\e)}$ and  $k \in N_t,$ where  $N_t : = \{k ; A_k \subset B(x, t( \ell_{\rho}(m _\rho) +\e)) \}.$ Then,
 \[ -\sum_{k \in N_t}  q_k^\rho (t) \log q_k^\rho (t) \leq \sum_{k \in N_t}  q_k^\rho (t) \x \log \#N_t.\]
 Since $\# N_t \leq C e^{t (\ell_{\rho} (m_\rho)+\e)(V+ \e)}$, for some constant $C$, Theorem  \ref{guivarc'h} follows.
 \end{proof}
 
 \begin{theo}\label{cuter} Assume that $\ell_{\rho} (m_\rho) >0.$ Then,  $ \int B \, dm_{\rho} \leq V$,  with equality in this  inequality only when $(M,g)$ is locally symmetric.\end{theo}
\begin{proof}  Recall equation (\ref{basic}): $ \int B \, dm_{\rho}  = \int  <\ov X,   \nabla ^s\log \mathsf{K}^\rho >\, dm_{\rho} ,$ so that, by Schwarz inequality,
\[ \left(\int B \, dm_{\rho}\right)^2 \leq \int_{SM}\|\nabla ^s \log  \mathsf{K}^\rho_{x, \xi} \|^2 \, dm_{\rho},\]
with equality only if $  \nabla ^s\log \mathsf{K}^\rho = \tau (\rho) \ov X$ for some real number $\tau (\rho).$  Abbreviate $h_{\rho}(m_{\rho})$, $\ell_{\rho}(m_{\rho})$ as $h_{\rho}, \ell_{\rho}$.   
We write  \[h_\rho = \int_{SM}\left( \|\nabla ^s \log  \mathsf{K}^\rho_{x, \xi} \|^2 - \rho B(x,\xi) \right) \, dm_{\rho} \geq \left( \int B \, dm_{\rho}\right)^2 - \rho \int B \, dm_{\rho}= \ell_\rho \int B\, dm_\rho.\] We indeed have $ \int B \, dm_{\rho} \leq V$, with equality only if $  \nabla ^s\log \mathsf{K}^\rho = \tau (\rho) \ov X$ for some real number $\tau (\rho).$  Then, equation  (\ref{int.byparts})  holds with $V$ replaced by $\tau (\rho)$. The proof of Corollary \ref{BRstrongstable} applies and the operator $\D^{ss}$ is symmetric  with respect to the measure $m_\rho$. By  Remark \ref{BRergodic}, $m_\rho = m_{BR}.$  Then, $\tau (\rho ) = V$  and from $\int B \, dm_\rho = \int B \, dm_{BR} = V$ and  $\ell_{\rho} (m_\rho) >0, $ we have $ \rho \neq V.$  
 We have $ 0= \LL_y ^{\rho \ast} e^{-V b_{x,\xi } (y)}\big|_{y = x} = (V-B(x, \xi))(V-\rho). $  It follows that $B= V$ is constant. By Theorem \ref{BCG}, the space $(M,g)$ is locally symmetric.
\end{proof} 

The conclusion in Theorem \ref{cuter} actually holds true for all $\rho<V$ due to the following.

\begin{prop}\label{rho<V} Let $\rho\in \Bbb R$. There is some $\LL^{\rho}$ stationary ergodic measure $m_{\rho}$ such that $\ell_{\rho} (m_\rho) >0$ if, and only if, $\rho<V$.  Moreover, the measures $m_\rho $ weak* converge to $m_{BR} $ as $\rho \nearrow V.$  
\end{prop}
\begin{proof} Let $\rho _0$ be such that there is some $\LL ^{\rho _0} $ stationary measure $m_{\rho _0} $  with $\ell_{\rho_0} (m_{\rho _0} ) \leq 0$,  but such that there exist $\{ \rho _n\}_{n \in \N}$  with $\lim _{n \to +\infty } \rho _n = \rho _0 $  and  $\ell_{\rho_n} (m_{\rho _n} ) > 0$  (we know that $m_{\rho _n} $ is unique by Corollary \ref{ergodicity}).  Observe that by equation (\ref{ell}), $\ell _{\rho } >0$ for $\rho $ sufficiently close to $-\infty$. On the other hand, if  $\ell _{\rho _n}(m_{\rho_n}) >0$, we must have  $\rho _n < V$  by equation (\ref{ell}) and Theorem \ref{cuter}. Therefore one can choose  $\rho _0$ and $\rho _n$ with those properties. 
Let $m$ be a weak* limit of the measures $m_{\rho _n}$.  We are going to show that $m = m_{BR}$ and that $\rho _0 = V$.

Observe that $\ell _{\rho _0} (m) \leq 0$ since otherwise $m$ is the only stationary measure and we cannot have $\ell _{\rho _0}(m_{\rho _0} ) \leq 0$   for some other $\LL ^{\rho _0} $ stationary measure $m_{\rho _0}$.  On the other hand, $\ell_{\rho_0} (m ) \geq 0$  by continuity, so  $\ell_{\rho_0} (m) = 0$ and $\lim_{n\to +\infty } \ell_{\rho_n} (m_{\rho _n})=0.$   By Theorem \ref{guivarc'h},  $\lim _{n\to +\infty }h_{\rho_n}(m_{\rho _n})=0 $ as well. 
 We have 
\begin{eqnarray*}  0 &= & \lim\limits _{n\to +\infty }h_{\rho_n}(m_{\rho _n}) \;=\;  \lim\limits_{n\to +\infty}  \int_{SM}\left( \|\nabla ^s \log  {\mathsf{K}}^{\rho_n}(x, \xi) \|^2 - \rho_n  B(x,\xi) \right) \, dm_{\rho_n} \\ &=&  \lim\limits_{n\to +\infty}  \int_{SM}\left( \|\nabla ^s \log  {\mathsf{K}}^{\rho_n}(x, \xi) \|^2 - \rho_n  <\ov X,   \nabla ^s\log \mathsf{K}^{\rho _n}>\right)\, dm_{\rho _n}.
\end{eqnarray*}
Write  $Z_n := \nabla ^s \log  {\mathsf{K}}^{\rho_n}(x, \xi) -\left( \int_{SM} <\ov X,   \nabla ^s\log \mathsf{K}^{\rho _n}>\, dm_{\rho _n}\right)\ov X.$   We have  
\begin{align*} & \lim\limits_{n \to +\infty } \int_{SM}  \|Z_n \|^2 \, dm_{\rho_n} = \lim\limits_{n\to +\infty}  \int_{SM}\left( \|\nabla ^s \log  {\mathsf{K}}^{\rho_n}(x, \xi) \|^2 \right)\, dm_{\rho _n} - \left( \int_{SM} B \, dm_{\rho _n}\right) ^2\\
 &\ \ = \lim\limits_{n\to +\infty} \left( h_{\rho_n} (m_{\rho _n}) - \ell_{\rho_n} (m_{\rho _n}) \int_{SM} B \, dm_{\rho _n}\right)\end{align*} 
and so $ \lim_{n \to +\infty } \int_{SM}  \|Z_n \|^2 \, dm_{\rho_n} =  0.$ In other words, equation  (\ref{int.byparts})  holds with $V$ replaced by $\int_{SM} B \, dm_{\rho _n}$ with an error $\int_{SM} < Z, Z_n> \, dm_{\rho _n}$. The proof of Corollary \ref{BRstrongstable} applies and the operator $\D^{ss}$ is symmetric  with respect to the measure $m_{\rho _n},$ up to an error which goes to 0 as $n \to +\infty.$ It follows that  the operator $\D^{ss}$ is symmetric  with respect to the limit measure $m.$ By  Remark \ref{BRergodic}, $m= m_{BR}.$ Since 
$\ell _{\rho _0}(m)= 0$, $\rho _0 = \int_{SM} B \, dm = \int_{SM} B \, dm_{BR} = V.$ \end{proof}

\begin{remark}Anderson and Schoen (\cite{AS}) described the Martin boundary for the Laplacian on a simply connected manifold with pinched negative curvature. Regularity of the Martin kernel in the \cite{AS} proof yields,  in the cocompact case, nice properties of  the harmonic measure (i.e., the stationary measure for $\LL ^0$). This was observed by  \cite{H}, \cite{Ka3} and \cite{L2}.  Ancona (\cite{Anc}) extended \cite{AS}'s results to the general weakly coercive operator and proved the basic inequality (\ref{Ancona}). This allowed Hamenst\"adt to consider the general case that  $\LL=\D^s +Y$, with $Y^*$, the dual of $Y$ in the cotangent bundle to the stable foliation over $SM$, satisfying  $dY^*=0$ leafwisely (\cite{H2}).   The  criterion she obtained for the existence of a $\LL$ stationary ergodic measure $m$ with $\ell_{\LL}(m):=\int_{M_0\times \partial\M} \left(-< Y, \overline{X}> +B\right)\ dm>0$ is $P(-<\ov X,Y>) >0$. Our presentation follows \cite{H2}, with a few simplifications when $Y = \rho  \ov X.$  Theorem \ref{cuter} was shown by Kaimanovich (\cite{Ka2}) in the case $\rho = 0$.  From \cite{H2}, Theorem A (2), the measure $m_{BR}$ is the only symmetric measure for $\LL^V$. It is not known whether $m_{BR}$ is the only stationary measure for $\LL^V$. The second statement in Proposition \ref{rho<V} would also follow from such a   uniqueness result.
\end{remark}

\

\section{ Stochastic flows of diffeomorphisms and a relative entropy.}
In this section, we introduce a stochastic flow associated to $\LL^\rho.$ In the case of $\rho = 0$ our object has been considered as a {\it{stochastic (analogue of) the geodesic flow}}  (cf. \cite{CE, El}). It gives rise to a random walk on the space of homeomorphisms of a bigger compact manifold and the relative entropy of this random walk of homeomorphisms is our   fourth  entropy. The upper semicontinuity of this entropy as $\rho \to -\infty $ will be used to prove that the measures $m_\rho$ converge to $m_L$ as $\rho \to -\infty $ (see Theorem \ref{cont.} below).

\subsection{Stochastic flow adapted to $\LL^{\rho}$}
Let $O\M$ be the  the orthonormal frame bundle (OFB) of $(\M, \wt g)$:
\[ O\M := \big\{ x \mapsto u(x):\  u(x) = (u^1, \cdots, u^d) \in O(S_x\M) \big\}\] 
and consider $ O\M \x \{\xi \} =: O^sS\M$, the OFB in $T \wt W^s$ and 
$O^sSM := O^sS\M / \G$, the OFB in $T W^s.$
For $v\in S\M, u \in O^s_vS\M$, the {\it { horizontal }} subspace of $T_uO^sS\M$ is the space of directions $w$ such that 
$\nabla _u w = 0$. 

Denote $D^r ( O^sS\M)$ $(r\in {\Bbb N}\ {\rm{or}}\ r=\infty)$  the  space of homeomorphisms $\Phi $ such that 
\[ \Phi (x, u, \xi ) := \left( \phi _\xi (x,u), \xi \right) ,\]
where $\phi_\xi $ is a $C^r $ diffeomorphism of $O\M$,  which depends continuously on $\xi$ in 
$\pp\M.$ 
We use  stochastic flow theory to define a random walk on $D^\infty ( O^sS\M)$.

\begin{theo}[\cite{El}]
 Let $(\Om, \P) $ be a  $\R^d$  Brownian motion (with covariance $2t {\mathbb {Id}}).$ 
For $\P$-a.e. $\om \in \Om,$ all $t >0$, there exists $\Phi ^\rho _t = \big (\phi ^\rho _{\xi ,t} , \xi \big)  \in D^\infty ( O^sS\M)$ such that 
for all $u \in O^sS\M, \, (\om, t) \mapsto u_t = \phi ^\rho _{\xi ,t} (u) $ solves the Stratonovich Stochastic Differential Equation (SDE)
\begin{equation}\label{SDE} du_t = \rho \wh X(u_t) + \sum_{i=1}^d \wh H (u^i_t) \circ dB^i_t , \end{equation}
where $\wh X, \wh H(u^i) $ are the horizontal lifts of $\ov X, u^i \in T_v\wt W^s(v) $ to $T_uO^sS\M$. 
Moreover,  \begin{enumerate} \item[1)] for $\P$-a.e. $\om \in \Om,$ all $t,s >0,$  $\rho <V,$    $ \xi \in \pp\M,$
\[ \phi ^\rho _{\xi ,t+s } (\om) = \phi ^\rho _{\xi ,t} (\s_s \om ) \circ \phi ^\rho _{\xi ,s} (\om),\]
where $\s _s$ is the shift on $\Om$,
\item[2)] for $\P$-a.e. $\om \in \Om ,$ for all $\b \in \G$, all $t >0$, $D\b \circ \phi ^\rho _{\xi ,t} (\om ) = \phi ^\rho _{\xi ,t}(\om) \circ D\b, $  and
\item[3)]  for $\P$-a.e. $\om \in \Om ,$ all $t >0$, $\rho \mapsto \Phi ^\rho _t (\om)$ is continuous in $D^\infty ( O^sS\M)$
and the derivatives are solutions to the  derivative SDE. 
\end{enumerate} \end{theo}

Relation (\ref{SDE}) implies  that for all $(x,\xi, u), u \in OS_x\M,$ the projection of $\phi ^\rho _{\xi ,t} (\om ) (u) $ on $S\M$ is a realization of the 
$\LL^\rho $ diffusion starting from $(x,\xi)$.

Property 1)  and independence of the  increments of the Brownian motion gives that if $\k_{\rho, s} $ is the distribution of $\Phi _{\rho, s} (\om)$ in 
$D^\infty ( O^sS\M)$, we can write
\[ \k_{\rho, s+t} = \k_{\rho, t} \ast \k_{\rho, s}, \]
where $\ast $ denotes the convolution in the group $D^\infty ( O^sS\M)$. So we have a {\it {stochastic flow}}. 
Property 2) yields a stochastic flow on $D^\infty ( O^sSM)$. Property 3) will allow  to control derivatives.

Fix $t>0$. A probability measure $\ov m$ on $O^sSM$ is said to be {\it {stationary }} for $\k _{\rho, t}$,  if for any $F \in C(O^sSM)$, the  set of continuous functions on $O^sSM$, \[ \int _{O^sSM} F(u) \, d\ov m(u) \;=\;
  \int _{D^\infty(O^sSM)}  \int _{O^sSM} F(\Phi u) \, d\k _{\rho, t} (\Phi) \, d\ov m(u).\]

\begin{prop}Fix any $\rho <V, t>0.$ The probability measure $\ov m_\rho$ on $O^sSM$ that  projects to $m_\rho $ on $SM$ and is the normalized Lebesgue measure on the fibers is stationary for $\k _{\rho, t} $. If  we write $O^sSM = O\M \x \pp \M$, then, up to a normalizing constant,   \[ d\ov m_\rho (x, u, \xi ) = d\nu^\rho_x (\xi ) d\vol (x, u).\] \end{prop}

\subsection{Entropy of a random transformation}
There is a notion of entropy for random transformations with a stationary measure (see \cite{KL} for details).

Let $X$ be a compact metric space and $D^0X$ the group of homeomorphisms of $X$. Let $\k$ be a probability measure on $D^0X$ and let $\ov m$ be a stationary measure for $\kappa.$ Let $\s $ be the shift on $(D^0X )^{\otimes \N}$, $\mathcal{K} = \k^{\otimes \N}$ the Bernoulli $\s$-invariant measure, $\ov \s$ the skew-product transformation on $(D^0X)^{\otimes \N} \x X$
\[ \ov \s (\un \phi, x) := (\s \un \phi , \phi _0 x ),\  {\forall} \un \phi = (\phi _0, \phi _1, \cdots ) \in (D^0X)^{\otimes \N}.\]
\begin{prop} Let $\ov m$ be a stationary measure for $\k$. Then, the measure $\mathcal{K}\times \ov m$  is $\ov \s$-invariant. \end{prop}

For $\un \phi \in (D^0X)^{\otimes \N}, x \in X, \e >0, n\in \N$, define a {\it {random Bowen ball}} by \[ \ov B(\un \phi, x, \e, n ) := 
\{ y:\ y \in X, d(\phi _k \circ \cdots\circ \phi _0 y,  \phi _k \circ \cdots \circ \phi _0 x) < \e, \ \forall 0 \leq k <n \} \] and the {\it {relative  entropy}} $h_{\ov m}(\mathcal{K})$  as the $\mathcal{K}$-a.e. value of 
\[\sup_\e \int_X  \limsup _{n\to +\infty} -\frac{1}{n} \log \ov m ( \ov B(\un \phi, x, \e, n ) ) \, d\ov m(x) .\] 
With the preceding  notations, take $X= O^sSM,$  $ \k = \k_{\rho, t}$ for some $(\rho, t), \rho <V, 0<t,$ and the stationary measure $\ov m_\rho$. We want to estimate the relative entropy $h_{\ov m_\rho} ( \mathcal{K}_{\rho,t}) $. 
\begin{prop}[\cite{LS4}]\label{randomPesin} We have 
\[ h_{\ov m_\rho} (\mathcal{K}_{\rho,t}) \geq \int \log \left| {\rm {Det}} D_u \Phi \big|_{T_u O^sS\M } \right| \, d\k_{\rho,t} (\Phi)\, d\ov m_\rho (u) .\]\end{prop}

Recall that $\ov m_\rho$ has  absolutely continuous conditional measures on the foliation $\ov {\W}^s$  defined by $(O\M \x \{\xi\})/\G.$ The proof uses ingredients from the proof of Pesin formula in the non-uniformly hyperbolic case (cf. \cite{Me}) and the non-invertible case (\cite{Li}, \cite{LiS}). Observe that, even if  $\Phi _{-1}\big|_{\ov {\W}^s} $  has only nonnegative  exponents, there might be negative exponents   for the random walk, and the inequality in  Proposition  \ref{randomPesin} might be strict.

\subsection{Continuity of the relative entropy}
We now indicate the main ideas of the proof of the following theorem
\begin{theo}[\cite {LS4}]\label{cont.} For $\rho <V$, let $m_\rho $ be the stationary measure for the diffusion on $SM$ with generator $\LL^\rho  = \D^s + \rho\ov X$.   Then, as $\rho \to -\infty ,$ $ m_\rho $ weak* converge to the Liouville measure $m_L$.\end{theo}

\begin{cor} $\lim\limits _{\rho \to -\infty } \int B \,dm_\rho \; = \; \int B\, dm_L \; = \; H.$ \end{cor}

\begin{proof}\label{cont.2} Set $\k_\rho = \k _{\rho, \frac{-1}{\rho}}.$ 
We first observe that as $\rho \to -\infty $, $\k_\rho$ weak* converge on $D^\infty (O^sSM)$ to the Dirac measure on the reverse frame flow $\Phi _{-1}$. Moreover, for any $r \in \N, r\geq 1$,
\[ \limsup _{\rho \to -\infty } C_r (\rho) < +\infty, {\textrm { where }} C_r(\rho ):= \int \|\Phi \|_{D^r(O^sSM)} \, d\k_\rho (\Phi ),\]
where $\|\cdot\|_{D^r(O^sSM)}$ is the supremum of leafwise $C^r$ norm.
Indeed, by definition, $\k_\rho$ is the distribution of the time one of the stochastic flow associated to the Stratonovich SDE
\[ du_t =-  \wh X(u_t) + \frac{-1}{\rho} \sum_{i=1}^d \wh H (u^i_t) \circ dB^i_t .\]
When $\rho \to -\infty $, the SDE converge to the ODE on  $O^sSM, \; du_t =-  \wh X(u_t) $. The convergence, and the control on $C_r$, follow by continuity of the solutions  in  $D^\infty (O^sSM)$.

Let then $m$ be a weak* limit of the measures $m_\rho $ as $\rho \to -\infty $,  $\ov m$ its extension to $O^sSM$ by the Lebesgue measure on the fibers. The measure $m$ is $\vf_{-1}$ invariant, $\ov m $ is the weak* limit of the measures $ \ov m_\rho $ and $\ov m $ is $\Phi _{-1}$ invariant. Moreover, $h_m (\vf _{-1}) = h_{\ov m} (\Phi _{-1})$ (this is a compact isometric extension) and  
\begin{align*}   \int \log \left| {\textrm {Det}} D_v \vf _{-1}\big|_{T_vW^s(v)} \right| \, dm (v) &= \int \log \left| {\textrm {Det}} D_u \Phi _{-1}\big|_{T_u O^sS\M} \right| \, d\ov m (u) \\ &= \lim\limits _{\rho \to -\infty } \int \log \left| {\textrm {Det}} D_u \Phi \big|_{T_u O^sS\M} \right| 
d\ov m_\rho (u) \, d\k_\rho (\Phi ).\end{align*}  By \cite{BR},  the Liouville measure is the only $\vf_{-1}$ invariant  measure with \[h_m (\vf _{-1}) = \int \log \left| {\textrm {Det}} D_v \vf _{-1}\big|_{T_vW^s(v)} \right| \, dm (v).\] To conclude the theorem, using Proposition \ref{randomPesin}, it suffices to show 
\[  h_{\ov m} (\Phi _{-1}) \geq \limsup _{\rho \to -\infty} h_{\ov m_\rho} ({\mathcal{K}}_{\rho}).\]
This will follow from the properties of the {\it{topological relative conditional entropy}} in the next subsection.
\end{proof}

\subsection{Topological relative conditional entropy}
The following definition extends the definition of Bowen (\cite{B}) to the random case, following Kifer-Yomdin (\cite{KY}) and Cowieson-Young (\cite{CY}). 

For $\e>0$ and $\un \phi \in (D^0X)^{\otimes \N}, x \in X, \tau >0, n\in \Bbb N$,  set $r(\e , \un \phi, x, \tau, n)$ for the smallest number of random $\ov B(\un \phi, y , \tau, n)$ balls needed to cover $\ov B(\un \phi, x, \e, n) $ and 
\[ h_{loc} (\e , \un \phi) := \sup _x \lim\limits_{\tau \to 0} \limsup _{n\to +\infty} \frac{1}{n} \log r(\e , \un \phi, x, \tau, n).\]
The function $\un \phi \mapsto  h_{loc} (\e , \un \phi) $ is $\s$-invariant. For $X = O^sSM$, write $h_{\rho,loc} (\e)$ for the ${\mathcal{K}}_\rho$-essential value of $ h_{loc} (\e , \un \phi).$
The conclusion follows from the two following facts (cf. \cite{LS4}, Section 4). 
\begin{prop}\label{Bowen} For all $\e >0$,  \[ h_{\ov m}(\Phi _{-1} )  \geq \limsup_{\rho \to -\infty }  h_{\ov m_\rho} ({\mathcal K}_{\rho}) -  \limsup_{\rho \to -\infty } h_{\rho,loc} (\e).\]\end{prop}
\begin{prop}\label{Yomdin} There is a constant $C$ such that, for all $r \in \N$, $r\geq 1,$  there is $\rho _r $ such that, for $\rho < \rho_r, $ \[ \lim\limits _{\e \to 0} \sup_{\rho <\rho_r} h_{\rho,loc} (\e)\leq \frac{C}{r} \, C_1,\]
where $C_1 = \sup_{\rho <\rho_1}  \int \|\Phi \|_{D^1(O^sSM)} \, d\k_\rho (\Phi ). $ \end{prop}
Proposition \ref{Bowen} in the deterministic case is due to Bowen (\cite{B}). Proposition \ref{Yomdin}  in the deterministic case is a famous result of Yomdin (\cite{Yo}, \cite{Y2}) and Buzzi (\cite{Buz}).  By proposition \ref{Yomdin}, since $r$ is arbitrary, $  \lim _{\e \to 0} \limsup_{\rho \to - \infty} h_{\rho,loc} (\e) =0.$  Proposition \ref{Bowen} then yields the claimed inequality.

\subsection{Conclusion. Katok's conjecture}
 Let $(M, g)$ be a $C^ \infty $  $d$-dimensional Riemannian manifold with negative curvature. We introduced in Sections 1 and 2  the numbers $H,$ the entropy of the Liouville measure
for the geodesic flow,  $V,$ the topological entropy of the geodesic flow,  and   the function $B$ on $SM$. The function $B$ is constant if, and only if $(M,g)$ is a locally symmetric space (Theorem \ref{BCG}). Using thermodynamical formalism, $H\leq V$ and  if $H=V,$ there exists a continuous function $F$ on $SM$, $C^1$ along the trajectories of the flow, such that 
$  B = V - \frac{\pp }{\pp t } F\circ \vf _t \big|_{t=0}$ (see Theorem \ref{cute}). Katok's conjecture (see surveys \cite{L5} and \cite{Yu} for  history of this topic) is that this can only happen when $(M,g)$ is a locally symmetric space, that is, when $B$ is constant on $SM$. This was proven by Katok (\cite{K2}) in dimension 2 and more generally if $g$ is conformally equivalent to a locally symmetric $g_0$. It was also proven by   Flaminio  (\cite{Fl}) in a $C^2$ neighborhood of a constant curvature metric $g_0$.    Here, we introduced a family of measures $m_\rho, \rho \leq V,$ such that 
$\int B \, dm_V = V$ and for $\rho <V, \int B \, dm_\rho \leq V$ with equality only in the case of locally symmetric spaces (Theorem \ref{cuter}). Finally, in the $C^\infty$ case, we also show that $\lim _{\rho \to -\infty} \int B \, dm_\rho = H$ (Corollary \ref{cont.2}).

\end{document}